\let\ORIlabel\label
\let\ORIrefstepcounter\refstepcounter
\AddToHook{package/hyperref/before}{%
  \let\label\ORIlabel
  \let\refstepcounter\ORIrefstepcounter
}
\documentclass[hidelinks,onefignum,onetabnum]{siamart220329}

\usepackage[normalem]{ulem}
\usepackage{amsfonts}
\usepackage{graphicx}
\usepackage[caption=false]{subfig} 
\usepackage{xcolor} 
\usepackage{algorithm,algorithmic} 
\usepackage{booktabs} 
\usepackage{placeins}
\usepackage{multirow}

\newcommand{\bx}{\mathbf{x}}
\newsiamremark{remark}{Remark}
\newsiamremark{example}{Example}
\newsiamremark{hypothesis}{Hypothesis}
\newsiamremark{assumption}{Assumption}
\crefname{hypothesis}{Hypothesis}{Hypotheses}
\newsiamthm{claim}{Claim}

\headers{Second-order flows via scalar auxiliary variables}{H.~Chen, W.~Liu, Z.~Xie, and Y.~Yuan}

\title{Second-Order Flows for Approaching Stationary Points via the SAV-E Approach\thanks{\textbf{Funding:} The work of H.~Chen, Z.~Xie, and Y.~Yuan was supported by the
National Natural Science Foundation of China (NSFC) under Grant
No.~12171148. Z.~Xie was also supported by the Foundation Sciences of Hunan Province (Grant No.~2024JC0006). H.~Chen was
partially supported by the Hubei Province Postdoctoral Talent
Introduction Project (No.~2024HBBHJD061) and the China Postdoctoral
Science Foundation (No.~2025M773049). The work of H.~Chen and W.~Liu
was supported by the NSFC under Grant No.~12571448. W.~Liu was also
supported by the Innovation Research Foundation of NUDT
(No.~202402-YJRC-XX-002).}}

\author{Haifan Chen\thanks{Institute for Math \& AI, Wuhan University, Wuhan 430072, China; 
  and MOE-LCSM, School of Mathematics and Statistics, Hunan Normal University, Changsha 410081, China
  (\email{haifan.chen@whu.edu.cn}).}
  \and Wei Liu\thanks{College of Science, National University of Defense Technology, Changsha 410073, China 
  (\email{wl@nudt.edu.cn}).}
\and Ziqing Xie\thanks{Corresponding author. MOE-LCSM, School of Mathematics and Statistics, Hunan Normal University, Changsha 410081, China
  (\email{ziqingxie@hunnu.edu.cn}).}
\and Yifang Yuan\thanks{MOE-LCSM, School of Mathematics and Statistics, Hunan Normal University, Changsha 410081, China
  (\email{19848082195@163.com}).}}

\usepackage{hyperref} 

\ifpdf
\hypersetup{
  pdftitle={Second-Order Flows for Approaching Stationary Points via the SAV-E Approach},
  pdfauthor={H. Chen, W. Liu, Z. Xie, and Y. Yuan}
}
\fi

\begin{document}

\maketitle

\begin{abstract}
Second-order flows represent an emerging computational framework for approaching stationary points of variational energies, characterized by damped hyperbolic partial differential equations. In this paper, we develop a scalar auxiliary variable (SAV) method for this class of flows, referred to as SAV-E. The resulting scheme is linear and decoupled, and we prove unconditional modified-energy stability. We further establish subsequential convergence to stationary points for the Ginzburg--Landau and conserved Swift--Hohenberg models under suitable assumptions. Numerical experiments support the analysis and demonstrate robustness and efficiency compared with representative baselines and the corresponding gradient flows.
\end{abstract}

\begin{keywords}
variational problem,
second-order flow, 
scalar auxiliary variable, 
stationary point,
Ginzburg--Landau energy, 
conserved Swift--Hohenberg model
\end{keywords}

\begin{MSCcodes}
35Q90, 35L71, 65M12, 65K10 
\end{MSCcodes}

\section{Introduction}\label{sect:intro}
Many fundamental phenomena in science and engineering are governed by variational principles, with stable equilibrium states characterized as minimizers of an associated free energy functional. The search for these minimizers, known as the variational minimization problem, is a central task with applications across diverse fields such as quantum physics, image processing, cell biology, and materials science \cite{BaoCai13, Can70, ChaVes01, GenPro93, Fetter2009RMP}. Mathematically, the objective is to solve
\begin{equation}\label{eq:min_problem}
\min_{u \in X} E(u),
\end{equation}
where $X$ is a suitable Sobolev space. This paper focuses on a broad class of energies composed of a quadratic term and a nonlinear potential, which generally take the form
\begin{equation}\label{eq:energy_functional}
E(u)=\frac12\langle\mathcal{L} u,u\rangle+\int_{\Omega}F(u)d\mathbf{x}.
\end{equation}
Such problems can pose significant computational challenges due to nonlinearity and even non-convexity. Here, $\Omega$ is the spatial domain, $\mathcal{L}$ is a symmetric non-negative linear operator, and $F$ is a nonlinear function, whose derivative is denoted by $f(u) := F'(u)$. Throughout this paper, $\langle \cdot, \cdot \rangle$ denotes the $L^2$ inner product on $\Omega$.

A classical approach to solving such variational problems is the gradient flow method \cite{CaiSheXu17, DuLiuWan04, DuLiuWan06, WanZhaZha21, XuJiaZha13}, which evolves \(u(t)\) via the dissipative dynamics
\begin{equation}
\label{eq:gradient-flow}
\partial_t u = -\mathcal{G} \mu,
\end{equation}
where \(\mu:=\delta E/\delta u=\mathcal Lu+f(u)\) is the chemical potential, and \(\mathcal{G}\) is a symmetric non-negative operator (e.g., \(I\) for the \(L^2\)-flow or \(-\Delta\) for the \(H^{-1}\)-flow). These methods mimic steepest descent in the energy landscape, ensuring energy dissipation and requiring only first-order variational information. However, they often suffer from slow convergence and trapping in local minima or saddle points. Separately, optimization-based approaches have also been developed for directly minimizing phase-field and Landau-type free-energy functionals \cite{BaoCheJia22,BaoCheJiaQiu24,JiaSiCheBao20,ZhoHuaJiaZha24}.

Second-order flows have recently emerged as a promising alternative to
gradient-flow methods \cite{CheDonIglLiuXie25,CheDonLiuXie23, DonHinZha21}. Motivated by the established connection between Nesterov's accelerated gradient method and second-order ODEs \cite{SuBoyCan16}, the second-order flow serves as a novel computational model for variational problems, introducing inertia to yield damped hyperbolic PDEs of the form
\begin{equation}
\label{eq:second-order-flow}
\partial_{tt} u + \eta(t) \partial_t u = -\mathcal{G} \mu,
\end{equation}
with damping \(\eta(t) > 0\) (e.g., constant or \(\alpha/t\), \(\alpha \geq 3\)). Recent studies have demonstrated the efficiency and robustness of second-order flows for challenging problems, from computing ground states of Bose-Einstein condensates \cite{CheDonLiuXie23} to finding stable liquid crystal configurations \cite{CheDonIglLiuXie25}. Some theoretical foundations for second-order flows have also been established via convex-splitting schemes, including well-posedness of the continuous flow and the subsequential convergence of iterative schemes to a stationary point \cite{CheDonIglLiuXie25}.

A key challenge in designing numerical discretizations for second-order flows is to balance computational efficiency and energy stability in the treatment of nonlinear terms. Our recent work has examined this trade-off through two approaches. In \cite{CheDonLiuXie23}, we developed a stabilized semi-implicit scheme that achieves high computational efficiency; however, its unconditional energy stability lacks a theoretical guarantee. In \cite{CheDonIglLiuXie25}, we studied a convex-splitting approach, which ensures unconditional stability and convergence but at the expense of solving a nonlinear system per iteration. This persistent trade-off motivates the development of new algorithms specifically designed for second-order flows that are linear, efficient, and energy-stable, with rigorous stationary-point convergence guarantees.

This paper addresses this challenge by developing a novel algorithm based on the scalar auxiliary variable (SAV) method, a powerful tool for constructing linear, unconditionally stable schemes~\cite{SheXuYan18,SheXuYan19}. SAV has been widely used to discretize gradient flows and related dissipative evolution equations~\cite{CheWan21,li2022efficient, QiaSunZhaZha19,WanHuaWan21,Yan21}. It has also been applied to energy minimization via imaginary-time gradient flows~\cite{WanZhoHu23,ZhaShe19} and to saddle-point search~\cite{GuWanZha24}. However, the classical SAV formulation, referred to here as SAV-F, may converge to a state that does not satisfy the original Euler--Lagrange equation. In~\cite{LiuSheZha23}, a minimization-oriented SAV formulation based on the full energy $E(u)$, referred to here as SAV-E, was introduced for finite-dimensional optimization. Furthermore,~\cite{LiuSheZha23} coupled this SAV-E formulation with a relaxation step and an adaptive time-stepping strategy, demonstrating empirically faster convergence than many gradient-based methods in finite-dimensional settings. It is worthwhile to point out that SAV-E is distinct from other full-energy SAV variants~\cite{HuaShe22,HuaSheYan20}, which are mainly designed for accurate approximation of gradient-flow dynamics.

Building on \cite{LiuSheZha23}, we extend the SAV-E strategy to the second-order flow framework in infinite-dimensional settings and develop the SF-SAV-E algorithm for approaching stationary points of variational energies. We incorporate a relaxation step to control the decay of the auxiliary variable and an adaptive time-stepping strategy to improve computational efficiency. The resulting relaxed scheme admits a linear and decoupled implementation, and we establish its unconditional modified-energy stability. Under suitable assumptions, we establish subsequential convergence to stationary points for the Ginzburg--Landau (GL) and conserved Swift--Hohenberg (CSH) models. Numerical experiments demonstrate the robustness and efficiency of the proposed method and provide supporting evidence for the key assumption used in the convergence analysis.

The remainder of this paper is organized as follows. In Section~\ref{sect:2}, 
we develop the SF-SAV-E scheme with relaxation and adaptive time stepping strategies. Section~\ref{sect:conditional-convergence} presents the subsequential convergence analysis for the GL and CSH models. Section~\ref{sect:numerics} reports numerical results, and Section~\ref{sect:conclusion} offers concluding remarks.

\section{SAV-based algorithm for second-order flows}\label{sect:2}
In this section, we construct an adaptive relaxed SF-SAV-E scheme for the second-order flow \eqref{eq:second-order-flow}, derive a linear and decoupled implementation, and establish its modified-energy stability.

\subsection{Continuous second-order flow and energy dissipation law}
We start by analyzing the energy dissipation properties of the continuous second-order flow model \eqref{eq:second-order-flow}, which will serve as a benchmark for constructing our energy-stable numerical scheme. Introducing the velocity \( v = \partial_t u \), the model can be written as the first-order system:
\begin{subequations}\label{eq:2nd_flow_system}
\begin{align}
    \partial_t u &= v, \label{eq:2nd_flow_u} \\
    \partial_t v &= -\eta(t)v - \mathcal{G}\mu, \label{eq:2nd_flow_v}
\end{align}
\end{subequations}
where \(\mu:=\delta E/\delta u=\mathcal Lu+f(u)\). Let $H\subset L^2(\Omega)$ be a closed subspace. We assume that $\mathcal G$ is symmetric and non-negative on its domain $\operatorname{dom}(\mathcal G)\subset L^2(\Omega)$ and that
\[
\mathcal G:\operatorname{dom}(\mathcal G)\cap H\to H
\]
is positive definite and invertible. We denote the corresponding inverse by $\mathcal G^{-1}$ and define
\[
\|w\|_{\mathcal G^{-1}}^2:=\langle w,\mathcal G^{-1}w\rangle,
\qquad w\in H.
\]
For this induced norm, we use the standard identity
\begin{equation}\label{eq:norm-identity}
    2\langle a-b,\mathcal G^{-1}a\rangle=\|a\|_{\mathcal G^{-1}}^2-\|b\|_{\mathcal G^{-1}}^2+\|a-b\|_{\mathcal G^{-1}}^2.
\end{equation}

For the solution $(u, v)$ to the system \eqref{eq:2nd_flow_system}, we define the total energy as 
\begin{equation}\label{eq:total_energy}
    \mathcal{E}(u, v) = E(u) + \frac{1}{2} \|v\|_{\mathcal{G}^{-1}}^2.
\end{equation}
For sufficiently regular solutions of \eqref{eq:2nd_flow_system},
differentiating the total energy and testing \eqref{eq:2nd_flow_v} with $\mathcal{G}^{-1}v$ yield
\begin{equation}\label{eq:continuous-energy-law}
\frac{d}{dt}\mathcal E(u,v)
=
-\eta(t)\|v\|_{\mathcal G^{-1}}^2
\le 0.
\end{equation}

\subsection{The SAV-E formulation for second-order flows}\label{sect. 2.2}
We first recall the classical SAV formulation~\cite{SheXuYan18,SheXuYan19}. Assume that the nonlinear part of the energy $\int_{\Omega}F(u)\,d\mathbf{x}$ is bounded from below. Choose $C>0$ such that $\int_\Omega F(u)\,dx+C>0$ and introduce $r_F(t)=\sqrt{\int_\Omega F(u)\,dx+C}$.
The energy can then be rewritten as
\[
E(u)=\frac12\langle\mathcal{L}u,u\rangle+r_F^2-C.
\]
Using $\partial_tu=v$ and differentiating the definition of $r_F(t)$ with respect to time, the second-order flow \eqref{eq:2nd_flow_system} can be equivalently reformulated as
\begin{align}\label{eq:sav_f_continuous}
\left\{\begin{aligned}
    &\partial_t u = v, \\
    &\partial_t v = -\eta(t)v -  \mathcal{G}\left(\mathcal{L}u+ \frac{r_F}{\sqrt{\int_{\Omega} F(u) \, d\mathbf{x} + C}}f(u)\right), \\
    &\frac{d\,r_F}{dt} = \frac{1}{2} \left\langle \frac{f(u)}{\sqrt{\int_{\Omega} F(u) \, d\mathbf{x} + C}}, v \right\rangle.
\end{aligned}\right.
\end{align}
A straightforward first-order semi-implicit discretization leads to what we denote as the SF-SAV-F scheme:
\begin{align}\label{eq:sf-sav-f-scheme}
\left\{
\begin{aligned}
    &\frac{u^{n+1} - u^n}{\tau} = v^{n+1}, \\[3pt]
    &\frac{v^{n+1} - v^n}{\tau} = - \eta^{n+1} v^{n+1} -\mathcal{G}\left(\mathcal{L} u^{n+1} + \frac{r_F^{n+1}}{\sqrt{\int_{\Omega} F(u^n) \, d\mathbf{x} + C}} f(u^n)\right), \\[3pt]
    &\frac{r_F^{n+1} - r_F^n}{\tau} = \frac{1}{2} \left\langle \frac{f(u^n)}{\sqrt{\int_{\Omega} F(u^n) \, d\mathbf{x} + C}}, v^{n+1} \right\rangle.
\end{aligned}
\right.
\end{align}
When the numerical sequence converges to a steady state $(u^\star,v^\star,r_F^\star)$, the first equation of \eqref{eq:sf-sav-f-scheme} gives $v^\star=0$, and the second equation yields
\begin{equation}\label{eq:sav-f-stationary}
\mathcal{G}\left(
\mathcal{L}u^\star
+
\frac{r_F^\star}
{\sqrt{\int_\Omega F(u^\star)\,d\mathbf{x}+C}}
f(u^\star)
\right)=0.
\end{equation}
If $r_F^\star/\sqrt{\int_\Omega F(u^\star)\,d\mathbf{x}+C}\neq 1$, a steady state satisfying \eqref{eq:sav-f-stationary} need not satisfy the original Euler--Lagrange equation, which presents a limitation when SAV-F is used for variational minimization. This behavior is illustrated numerically in Example~\ref{eg:4-2}.

Following the minimization-oriented SAV strategy in~\cite{LiuSheZha23}, we instead assume that the full energy $E(u)$ is bounded from below. Choose $C>0$ sufficiently large such that $E(u)+C>0$ and introduce $r(t)=\sqrt{E(u)+C}$. Then \[
E(u)=r^2-C.
\] Using $\partial_tu=v$ and differentiating the definition of $r(t)$
with respect to time, the second-order flow \eqref{eq:2nd_flow_system} becomes
\begin{align}\label{eq:sav_e_continuous}
\left\{\begin{aligned}
    &\partial_t u = v, \\
    &\partial_t v = -\eta(t)v - \frac{r}{\sqrt{E(u) + C}}\, \mathcal{G}\left( {\mathcal{L} u + f\left(u\right)}\right), \\
    &\frac{dr}{dt} = \frac{1}{2\sqrt{E(u) + C}} \left\langle \mathcal{L} u + f(u), v \right\rangle.
\end{aligned}\right.
\end{align}
 For this system, we propose the following semi-implicit scheme, which includes a symmetric non-negative linear stabilization operator $\mathcal{S}$ to enhance stability:
\begin{align}\label{eq:sf-sav-e-scheme}
\left\{\begin{aligned}
    &\frac{u^{n+1} - u^n}{\tau} = v^{n+1}, \\
    &\frac{v^{n+1} - v^n}{\tau} = -\eta^{n+1} v^{n+1} - \mathcal{S} u^{n+1} -\frac{r^{n+1}}{\sqrt{E(u^n) + C}} \mathcal{G} \left( {\mathcal{L} u^n + f(u^n)}\right) + \mathcal{S} u^n, \\
    &\frac{r^{n+1} - r^n}{\tau} = \frac{1}{2\sqrt{E(u^n) + C}} \left\langle \mathcal{L} u^n + f(u^n), v^{n+1} \right\rangle.
\end{aligned}\right.
\end{align}
This scheme \eqref{eq:sf-sav-e-scheme} is linear and admits the decoupled implementation given in Section~\ref{subsec:relaxed-decoupled}. More importantly, if the three sequences $\{u^n\}$, $\{v^n\}$, and $\{r^n\}$ converge to $u^\star$, $v^\star$, and $r^\star$, respectively, the first equation of \eqref{eq:sf-sav-e-scheme} also gives $v^\star=0$, while the second equation yields
\begin{align}\label{eq:sav-e-stationary}
\frac{r^\star}{\sqrt{E(u^\star)+C}}\,
\mathcal{G}\bigl(\mathcal{L}u^\star+f(u^\star)\bigr)=0.
\end{align}
If $r^\star\neq0$, it follows that $\mathcal{G}\left(\mathcal{L}u^\star+f(u^\star)\right)=0$, which represents the stationary condition for the original variational problem. In particular, for the GL model with $\mathcal G=I$, this is the Euler--Lagrange equation $\mathcal L u^\star+f(u^\star)=0$, whereas for the mass-conserved CSH model with $\mathcal G=-\Delta$, it corresponds to stationarity under the mass constraint.

For the analysis below, we assume that the symmetric stabilization operator $\mathcal S$ satisfies
\begin{equation}\label{eq:compatibility}
\langle \mathcal S w,\,\mathcal G^{-1}w\rangle \ge 0
\qquad \text{for all admissible } w.
\end{equation}
A standard choice in practice is $\mathcal S=\mathcal G\mathcal L$, for which \eqref{eq:compatibility} reduces to \(\langle \mathcal L w,w\rangle\ge0\), consistent with the positive semi-definiteness of \(\mathcal L\).

\subsection{Relaxed SF-SAV-E and decoupled implementation}\label{subsec:relaxed-decoupled}
As shown in Section~\ref{sect. 2.2}, a limiting steady state of SAV-E with $r^\star\neq0$ satisfies the stationary condition for the original variational problem. In practice, however, the discrete auxiliary variable $r^n$ may deviate from $\sqrt{E(u^n)+C}$. Specifically, \(r^n\) may decay much faster than \(\sqrt{E(u^n)+C}\) and become very small before \(u^n\) approaches stationarity, which may cause the iteration to stagnate. To address this issue and enhance robustness, we incorporate a relaxation technique inspired by \cite{JiaZhaZha22,ZhaShe22,Zha21}. To accommodate the adaptive time-stepping strategy detailed in Section~\ref{subsec:adaptive}, we present the relaxed scheme in its general form with a variable time step \(\tau_{n+1}>0\):
\begin{equation}
\label{eq:relaxed-sav-scheme}
\left\{\begin{aligned}
& \frac{u^{n+1} - u^n}{\tau_{n+1}} = v^{n+1}, \\
& \frac{v^{n+1} - v^n}{\tau_{n+1}} = - \eta^{n+1} v^{n+1} - \mathcal{S} u^{n+1} -\frac{\tilde{r}^{n+1}}{\sqrt{E(u^n) + C}} \mathcal{G} \left( {\mathcal{L} u^n + f(u^n)}\right) + \mathcal{S} u^n, \\
& \frac{\tilde{r}^{n+1} - r^n}{\tau_{n+1}} = \frac{1}{2\sqrt{E(u^n) + C}} \left\langle \mathcal{L} u^n + f(u^n), v^{n+1} \right\rangle, \\
& r^{n+1} = \xi\, \tilde{r}^{n+1} + (1 - \xi) \sqrt{E(u^{n+1}) + C},
\end{aligned}\right.
\end{equation}
 where $\xi \in [0,1]$ is a relaxation parameter to be determined. The first three equations constitute the variable time-step formulation of the basic SF-SAV-E scheme \eqref{eq:sf-sav-e-scheme}, yielding an intermediate auxiliary variable $\tilde{r}^{n+1}$. The last equation for $r^{n+1}$ then adjusts this value via a convex combination with its theoretical value $\sqrt{E(u^{n+1}) + C}$. When \(\tau_{n+1}=\tau\) and \(\xi=1\), \eqref{eq:relaxed-sav-scheme} reduces to the basic SF-SAV-E scheme \eqref{eq:sf-sav-e-scheme}. The choice of \(\xi\) will be specified in Section \ref{subsec:relaxed-stability} so as to preserve the modified-energy dissipation property. We first derive a linear and decoupled implementation of the intermediate update.

Define
\[
q^n:=
\frac{\mathcal L u^n+f(u^n)}
{\sqrt{E(u^n)+C}},
\qquad
\mathcal A_{\tau,\eta}:=(1+\tau\eta)I+\tau^2\mathcal S .
\]
Eliminating $v^{n+1}$ from the first two equations of \eqref{eq:relaxed-sav-scheme} gives
\begin{equation}\label{eq:u_decomposition}
u^{n+1} = u_1^{n+1} + \tilde{r}^{n+1} u_2^{n+1},
\end{equation}
where $u_1^{n+1}$ and $u_2^{n+1}$ are the solutions of
\begin{equation}\label{eq:decoupled-subsystems}
\begin{aligned}
\mathcal{A}_{\tau_{n+1},\eta^{n+1}} u_1^{n+1} &= (1 + \tau_{n+1} \eta^{n+1}) u^n + \tau_{n+1} v^n + \tau_{n+1}^2 \mathcal{S} u^n, \\
\mathcal{A}_{\tau_{n+1},\eta^{n+1}} u_2^{n+1} &= -\tau_{n+1}^2 \mathcal{G} q^n.
\end{aligned}
\end{equation}
For the two representative models studied below, we take
$\mathcal S=-\Delta$ for the GL model and
$\mathcal S=-\Delta(\Delta+1)^2$ for the CSH model.
In both cases, $\mathcal A_{\tau,\eta}$ is an invertible
constant-coefficient linear operator for any $\tau,\eta>0$.
Hence the two subsystems in \eqref{eq:decoupled-subsystems}
admit unique solutions and can be solved efficiently using standard fast solvers, such as FFT-based methods under periodic boundary conditions. 

Substituting \eqref{eq:u_decomposition} together with
$v^{n+1}=(u^{n+1}-u^n)/\tau_{n+1}$ into the third equation of
\eqref{eq:relaxed-sav-scheme} gives
\begin{equation}\label{eq:r_scalar_relation}
\bigl(2-\langle q^n,u_2^{n+1}\rangle\bigr)\tilde r^{n+1}
=
2r^n+\langle q^n,u_1^{n+1}-u^n\rangle.
\end{equation}
Testing the second equation of \eqref{eq:decoupled-subsystems} with $\mathcal{G}^{-1}u_2^{n+1}$, we obtain
\[
-\tau_{n+1}^2\langle q^n,u_2^{n+1}\rangle=(1+\tau_{n+1}\eta^{n+1})\|u_2^{n+1}\|_{\mathcal{G}^{-1}}^2+\tau_{n+1}^2\langle\mathcal{S}u_2^{n+1},\mathcal{G}^{-1}u_2^{n+1}\rangle.
\]
By the compatibility condition \eqref{eq:compatibility}, the right-hand side is non-negative. Hence $\langle q^n,u_2^{n+1}\rangle\le0$, and therefore $2-\langle q^n,u_2^{n+1}\rangle\ge2>0$. Therefore,
\begin{equation}
\label{eq:r_update_formula}
\widetilde r^{n+1}
=
\frac{
2r^n+\langle q^n,u_1^{n+1}-u^n\rangle
}{
2-\langle q^n,u_2^{n+1}\rangle
},
\end{equation}
and $\widetilde r^{n+1}$ is uniquely determined. Thus, for any prescribed \(\xi\in[0,1]\), the relaxed one-step update is uniquely determined for the GL and CSH settings considered in this paper.

\subsection{Modified-energy stability and relaxation parameter selection}\label{subsec:relaxed-stability}
We first characterize the admissible choices of the relaxation parameter $\xi$ that preserve the modified-energy dissipation property. Define
\begin{equation}
\label{eq:xi-set}
\mathcal{V}_n = \left\{ \xi \in [0,1] \mid (r^{n+1})^2 - (\tilde{r}^{n+1})^2 - (\tilde{r}^{n+1} - r^n)^2 \leq \vartheta \mathcal{D}(v^{n+1}, v^n) \right\},
\end{equation}
where $\vartheta \in [0,1)$ is a user-defined parameter (set to 0.99 in our experiments), and
\begin{equation}\label{eq:mathcal-D}
\begin{aligned}
\mathcal{D}(v^{n+1}, v^{n})
:={}&\tau_{n+1}\,\eta^{n+1}\|v^{n+1}\|_{\mathcal{G}^{-1}}^2
+\tau_{n+1}^2
\langle \mathcal{S}v^{n+1}, \mathcal{G}^{-1}v^{n+1} \rangle \\
&\quad +\frac12 \|v^{n+1}-v^n\|_{\mathcal{G}^{-1}}^2 .
\end{aligned}
\end{equation}
By \(\eta^{n+1}>0\) and the compatibility condition \eqref{eq:compatibility}, we have \(\mathcal D(v^{n+1},v^n)\ge0\). Moreover, \(\mathcal V_n\) is nonempty since \(1\in\mathcal V_n\). Indeed, setting \(\xi=1\) gives \(r^{n+1}=\tilde r^{\,n+1}\), and hence \(-(\tilde r^{\,n+1}-r^n)^2\le \vartheta\mathcal D(v^{n+1},v^n)\).

The following theorem shows that any $\xi\in\mathcal V_n$
preserves the modified-energy dissipation property.
\begin{theorem}[Modified-energy stability]
\label{thm:energy-stability-relax}
For any $\tau_{n+1}>0$ and $\xi\in\mathcal V_n$, the relaxed scheme
\eqref{eq:relaxed-sav-scheme} satisfies
\begin{equation}
\label{eq:discrete_energy_law-relax}
\bar{\mathcal E}^{n+1}-\bar{\mathcal E}^{n}
\le
-(1-\vartheta)\mathcal D(v^{n+1},v^n)
\le 0,
\end{equation}
where
\[
\bar{\mathcal E}^{n}
:=
(r^n)^2+\frac12\|v^n\|_{\mathcal G^{-1}}^2.
\]
Moreover, the intermediate variables satisfy
\begin{equation}
\label{eq:intermediate-energy-estimate}
(\tilde r^{n+1})^2
+\frac12\|v^{n+1}\|_{\mathcal G^{-1}}^2
\le
\bar{\mathcal E}^{n}
\le
\bar{\mathcal E}^{0}.
\end{equation}
\end{theorem}

\begin{proof}
Multiplying the second equation of \eqref{eq:relaxed-sav-scheme} by \(\tau_{n+1}\) and testing it with \(\mathcal G^{-1}v^{n+1}\), and using \(u^{n+1}-u^n=\tau_{n+1}v^{n+1}\) in the stabilization term, we obtain
\[
\begin{aligned}
&\langle v^{n+1}-v^n,\mathcal G^{-1}v^{n+1}\rangle
+\tau_{n+1}\eta^{n+1}\|v^{n+1}\|_{\mathcal G^{-1}}^2
+\tau_{n+1}^2
\langle\mathcal S v^{n+1},\mathcal G^{-1}v^{n+1}\rangle\\
&\qquad
= -\tau_{n+1}\tilde r^{\,n+1}
\left\langle
\frac{\mathcal{L} u^n + f(u^n)}{\sqrt{E(u^n)+C}},
v^{n+1}
\right\rangle .
\end{aligned}
\]
Multiplying the third equation of \eqref{eq:relaxed-sav-scheme} by \(2\tau_{n+1}\tilde r^{\,n+1}\) and combining it with the identity above yields
\[
\begin{aligned}
2\tilde r^{\,n+1}(\tilde r^{\,n+1}-r^n)
&=
-\langle v^{n+1}-v^n,\mathcal G^{-1}v^{n+1}\rangle
-\tau_{n+1}\eta^{n+1}
\|v^{n+1}\|_{\mathcal G^{-1}}^2\\
&\quad
-\tau_{n+1}^2
\langle\mathcal S v^{n+1},\mathcal G^{-1}v^{n+1}\rangle .
\end{aligned}
\]
Using the identity \eqref{eq:norm-identity} and the definition of \(\mathcal D\) in \eqref{eq:mathcal-D}, we have
\begin{align}\label{eq:energy-identity}
\frac12\left(\|v^{n+1}\|_{\mathcal G^{-1}}^2-\|v^n\|_{\mathcal G^{-1}}^2\right)
&=-\mathcal D(v^{n+1},v^n)-(\tilde r^{\,n+1})^2+(r^{n})^2-(\tilde r^{\,n+1}-r^n)^2.
\end{align}
Rearranging the above identity gives
\begin{equation}
\label{eq:intermediate-step-estimate}
(\tilde r^{n+1})^2
+\frac12\|v^{n+1}\|_{\mathcal G^{-1}}^2
=
\bar{\mathcal E}^{n}
-\mathcal D(v^{n+1},v^n)
-(\tilde r^{n+1}-r^n)^2
\le
\bar{\mathcal E}^{n}.
\end{equation}
Adding \((r^{n+1})^2-(r^n)^2\) to both sides of \eqref{eq:energy-identity} gives
\[
\begin{aligned}
\bar{\mathcal E}^{n+1}-\bar{\mathcal E}^{n}
=
-
\mathcal D(v^{n+1},v^n)
+
(r^{n+1})^2
-
(\tilde r^{\,n+1})^2
-
(\tilde r^{\,n+1}-r^n)^2.
\end{aligned}
\]
By the definition of the admissible set \(\mathcal V_n\) in
\eqref{eq:xi-set}, we have
\[
\bar{\mathcal E}^{n+1}-\bar{\mathcal E}^{n}
\le
-(1-\vartheta)\mathcal D(v^{n+1},v^n).
\]
Since \(\vartheta\in[0,1)\) and
\(\mathcal D(v^{n+1},v^n)\ge0\), \eqref{eq:discrete_energy_law-relax} follows. Iterating this inequality yields $\bar{\mathcal E}^{n}\le\bar{\mathcal E}^{0}$. Together with \eqref{eq:intermediate-step-estimate}, this gives \eqref{eq:intermediate-energy-estimate}.
\end{proof}

We next specify a way to choose $\xi$. Define the reference value
\[
r_E^{n+1}:=\sqrt{E(u^{n+1})+C}.
\]
Since the purpose of the relaxation is to prevent the auxiliary variable from decaying too rapidly compared to its reference value, we apply it only when $\tilde r^{n+1}<r_E^{n+1}$. Otherwise, we set $\xi=1$.

For $\tilde r^{n+1}<r_E^{n+1}$, substituting $r^{n+1}
=\xi\tilde r^{n+1}+(1-\xi)r_E^{n+1}$ into the defining inequality of $\mathcal V_n$ yields
\begin{equation}\label{eq:quadratic_inequality}
a\xi^2+b\xi+c\le0,
\end{equation}
with coefficients $a = (\tilde{r}^{n+1}-r_E^{n+1})^2$, $b = 2(\tilde{r}^{n+1}-r_E^{n+1})r_E^{n+1}$ and
\begin{align*}
    c = (r_E^{n+1})^2 -(\tilde{r}^{n+1})^2-(\tilde{r}^{n+1}-r^n)^2-\vartheta\mathcal{D}(v^{n+1}, v^{n}).
\end{align*}
Since $\tilde r^{n+1}<r_E^{n+1}$, we have $a>0$. Moreover,
\[
b^2-4ac=4(\tilde r^{n+1}-r_E^{n+1})^2
\left[
(\tilde r^{n+1})^2
+(\tilde r^{n+1}-r^n)^2
+\vartheta\mathcal D(v^{n+1},v^n)
\right]\ge0.
\]
Hence the associated quadratic equation has two real roots $\xi_1$ and $\xi_2$, possibly coincident, with
$\xi_1\le\xi_2$, where
\begin{align}\label{eq:xi-calculation}
\xi_{1,2} = \frac{r_E^{n+1} \mp \sqrt{(\tilde{r}^{n+1})^2 + (\tilde{r}^{n+1} - r^n)^2 + \vartheta \mathcal{D}(v^{n+1}, v^{n})}}{r_E^{n+1} - \tilde{r}^{n+1}}.
\end{align}
The solution set of the quadratic inequality is therefore
$[\xi_1,\xi_2]$. Since $\xi=1$ satisfies
\eqref{eq:quadratic_inequality}, we have \([\xi_1,\xi_2]\cap[0,1]\neq\emptyset\). Moreover, for $\tilde r^{n+1}<r_E^{n+1}$, \([\xi_1,\xi_2]\cap[0,1]\subseteq\mathcal V_n\). Following \cite{JiaZhaZha22,LiuSheZha23}, we choose 
\begin{equation}\label{eq:xi-choice}
\xi=
\begin{cases}
1, & \tilde r^{n+1}\ge r_E^{n+1},\\[1mm]
\max\{0,\xi_1\}, & \tilde r^{n+1}<r_E^{n+1}.
\end{cases}
\end{equation}
In the second case, $\max\{0,\xi_1\}\in[\xi_1,\xi_2]\cap[0,1]$.
Thus the above choice belongs to $\mathcal V_n$.

The effectiveness of this relaxation strategy is illustrated numerically in Example~\ref{eg:4-1}.

\subsection{Adaptive time-stepping and algorithm}\label{subsec:adaptive}
The modified-energy stability established above permits a flexible choice of time steps. Following \cite{LiuSheZha23}, we use
\[
I^n:=\frac{r^n}{\sqrt{E(u^n)+C}}
\]
as an indicator for adaptive time stepping. Given $I_{\rm thr}\in(0,1)$, $0<\tau_{\min}\le\tau_0\le\tau_{\max}<\infty$, and a growth factor $\rho>1$, the time step is reduced when $I^n<I_{\rm thr}$ and
increased otherwise, while remaining in $[\tau_{\min},\tau_{\max}]$. The complete adaptive relaxed SF-SAV-E algorithm is summarized in Algorithm~\ref{alg:final_scheme}.

\begin{algorithm}[!htbp]
\caption{An Adaptive Relaxed SF-SAV-E Algorithm}
\label{alg:final_scheme}
\begin{algorithmic}[1]
    \STATE \textbf{Inputs:} \parbox[t]{0.82\linewidth}{Initial state $(u^0, v^0)$; parameters $0<\tau_{\min}\le\tau_0\le\tau_{\max}$, $\rho, I_{\text{thr}}, C, \vartheta$; damping function $\eta(t)$.}
    \STATE \textbf{Initialize:} $r^0 \gets \sqrt{E(u^0) + C}$, $t^0 \gets 0$, $n \gets 0$.
    \WHILE{the stopping criterion is not satisfied}
        \STATE \textit{\# 1. Adapt Time Step} \vspace{2mm}
        \STATE Compute the ratio indicator: $I^n=r^n / \sqrt{E(u^{n}) + C}$
        \vspace{0.5mm}
        \IF{\( I^n < I_{\text{thr}}\)}
            \STATE \(\tau_{n+1} = \max\{I^n \tau_n, \tau_{\min} \}\).
        \ELSE
            \STATE \(\tau_{n+1} = \min\{\rho\tau_n,\tau_{\max}\}\).
        \ENDIF
        \STATE Set $t^{n+1}=t^n+\tau_{n+1}$ and $\eta^{n+1}=\eta(t^{n+1})$.
\vspace{2mm}
\STATE \textit{\# 2. Compute Intermediate State $(u^{n+1}, v^{n+1}, \tilde{r}^{n+1})$}
\vspace{1mm}
\STATE Solve the decoupled systems for $u_1^{n+1}$ and $u_2^{n+1}$ via Eq. \eqref{eq:decoupled-subsystems}.
\vspace{0.5mm}
\STATE Compute $\tilde{r}^{n+1}$ via Eq. \eqref{eq:r_update_formula}.
\vspace{0.5mm}
\STATE Update $u^{n+1} = u_1^{n+1} + \tilde{r}^{n+1} u_2^{n+1}$ and $v^{n+1} = (u^{n+1} - u^n)/\tau_{n+1}$.
\vspace{2mm}
\STATE \textit{\# 3. Apply Relaxation Step}
\vspace{1mm}
\IF{$\sqrt{E(u^{n+1}) + C} - \tilde{r}^{n+1} > 0$}
    \STATE Compute $\xi = \max\{0, \xi_1\}$ according to \eqref{eq:xi-calculation}.
\ELSE
    \STATE $\xi = 1$
\ENDIF
        \STATE Update $r^{n+1} = \xi \tilde{r}^{n+1} + (1 - \xi) \sqrt{E(u^{n+1}) + C}$.
        \STATE $n := n + 1$
    \ENDWHILE
    \RETURN $u^n$
\end{algorithmic}
\end{algorithm}

\section{Subsequential convergence to stationary points}\label{sect:conditional-convergence}

In this section, we establish subsequential convergence for the two representative models considered in this paper, namely the GL and CSH models. Throughout this section, we consider the adaptive time steps generated by Algorithm~\ref{alg:final_scheme}, which satisfy
\begin{equation}
\label{eq:time-step-bounds}
0<\tau_{\min}
\le
\tau_{n+1}
\le
\tau_{\max}
<\infty,
\qquad n\ge0.
\end{equation}

We first introduce a ratio-control hypothesis for the relaxed SF-SAV-E scheme, whose numerical support will be illustrated in Example~\ref{eg:hypothesis-ratio-evidence}.
\subsection{A ratio-control hypothesis}
\label{subsec:ratio-roadmap}
\begin{hypothesis}[Ratio control]
\label{hyp:ratio-control}
Let $\{(u^n, v^n, r^n)\}_{n \ge 0}$ be the numerical sequence generated by the adaptive relaxed SF-SAV-E scheme, and let $\tilde{r}^{n+1}$ be the associated intermediate auxiliary variable. There exist an index \(N\in\mathbb N\) and a constant \(\delta\in(0,1)\) such that, for all \(n\ge N\), the ratio
\[
a^n := \frac{\tilde r^{n+1}}{\sqrt{E(u^n)+C}} \in [\delta,\, 2-\delta].
\]
\end{hypothesis}

\begin{remark}
The ratio in Hypothesis~\ref{hyp:ratio-control} can be decomposed as
\[
a^n=I^n+B^n,\qquad
I^n:=\frac{r^n}{\sqrt{E(u^n)+C}},\qquad
B^n:=\frac{\widetilde r^{n+1}-r^n}{\sqrt{E(u^n)+C}}.
\]
Here $I^n$ is the indicator used in adaptive time stepping, while $B^n$ accounts for the increment of the auxiliary variable. Hypothesis~\ref{hyp:ratio-control} is examined numerically in Example~\ref{ex:adaptive-ratio-gl}.
\end{remark}

\begin{lemma}[Original-energy bound under ratio control]
\label{lem:universal-physical-energy-bound}
Assume that the modified energy dissipation law holds and that Hypothesis \ref{hyp:ratio-control} is satisfied. Then the original energy \(E(u^n)\) is uniformly bounded from above for all \(n\in\mathbb N\).
\end{lemma}

\begin{proof}
For all $n\ge N$, Hypothesis~\ref{hyp:ratio-control} gives $a^n\ge\delta>0$. Combining this with \eqref{eq:intermediate-energy-estimate}, we obtain
\[
E(u^n)+C
=
\left(\frac{\widetilde r^{n+1}}{a^n}\right)^2
\le
\frac{\bar{\mathcal E}^0}{\delta^2},
\qquad n\ge N.
\]
Together with the finitely many values for \(n<N\), this yields $\sup_{n\ge0}E(u^n)<\infty$.
\end{proof}

\subsection{Subsequential convergence to stationary points for the GL model}\label{subsec:GL-convergence}

We first consider the GL model. For $d\le3$, we consider either a periodic box with periodic boundary conditions or a bounded $C^{1,1}$ domain with homogeneous Dirichlet or Neumann boundary conditions. We work with one of these boundary conditions throughout this subsection and let $V$ denote the corresponding test space $H^1_{\rm per}(\Omega)$, $H_0^1(\Omega)$, or $H^1(\Omega)$, respectively. Here \(H^1_{\rm per}(\Omega)\) denotes the periodic \(H^1\) space.

The GL energy is
\begin{equation}
E(u)=\int_{\Omega}\left(\frac12|\nabla u|^2+\frac{1}{4\varepsilon^2}(u^2-1)^2\right)\,d\bx,
\label{eq:GL-energy-convergence}
\end{equation}
where $\varepsilon>0$ is a model parameter. In this case,
\[
\mathcal L=-\Delta,\qquad \mathcal G=I,\qquad
f(u)=\frac{1}{\varepsilon^2}(u^3-u).
\]
For consistency with the numerical experiments in Section~\ref{subsec:GL}, we take
\[
\mathcal S=\mathcal G\mathcal L=-\Delta.
\]
We first record two model-specific estimates that will be used repeatedly below.

\begin{proposition}[Basic estimates for the GL model]\label{prop:GL-basic-estimates}
Let $E$ be given by \eqref{eq:GL-energy-convergence}. Then the following statements hold.
\begin{itemize}
    \item[(i)] \textbf{Energy coercivity.}
    There exist constants \(c_0>0\) and \(C_0\ge 0\) such that
    \[
    E(u)+C_0 \ge c_0 \|u\|_{H^1}^2
    \qquad \text{for all } u\in H^1(\Omega).
    \]

    \item[(ii)] \textbf{Nonlinear mapping estimate.}
    If \(\{u^n\}_{n\in\mathbb N}\) is uniformly bounded in \(H^1(\Omega)\), then
    \(\{f(u^n)\}_{n\in\mathbb N}\) is uniformly bounded in \(L^2(\Omega)\).
\end{itemize}
\end{proposition}
\begin{proof}
For part~\textnormal{(i)}, the pointwise inequality
\[
(s^2-1)^2\ge \frac12s^2-1,\qquad s\in\mathbb R,
\]
gives
\[
E(u)+\frac{|\Omega|}{4\varepsilon^2}
\ge \frac12\|\nabla u\|_{L^2}^2
   +\frac{1}{8\varepsilon^2}\|u\|_{L^2}^2
\ge c_0\|u\|_{H^1}^2,
\]
where $c_0:=\min\{1/2,1/(8\varepsilon^2)\}>0$.

For part (ii), we use the Sobolev embedding $H^1(\Omega)\hookrightarrow L^6(\Omega)$ for $d\le3$. It follows that
\[
\|f(u)\|_{L^2}
\le
\frac{1}{\varepsilon^2}
\left(
\|u\|_{L^6}^3+\|u\|_{L^2}
\right)
\le
\frac{C_\Omega}{\varepsilon^2}
\left(
\|u\|_{H^1}^3+\|u\|_{H^1}
\right),
\]
for some $C_\Omega>0$. The asserted uniform $L^2$ bound follows.
\end{proof}

\begin{lemma}[Conditional boundedness for the GL model]\label{lem:GL-boundedness}
Let \(\{u^n,v^n,r^n\}\) be the sequence generated by the adaptive relaxed SF-SAV-E scheme for the GL energy \eqref{eq:GL-energy-convergence}, with initial data \(u^0\in H^2(\Omega)\) satisfying the fixed boundary condition and \(v^0\in L^2(\Omega)\). Suppose Hypothesis~\ref{hyp:ratio-control} holds. Assume that the damping sequence satisfies
\[
\frac{\omega_0}{n+1}\le \eta^{n+1}\le \omega_1
\qquad \text{for some constants } \omega_0,\omega_1>0,
\]
for all \(n\ge0\). This condition is satisfied by both constant damping and the choice \(\eta(t)=\alpha/t\). Then the following statements hold.
\begin{itemize}
    \item[(i)] The sequence $\{v^n\}_{n\ge1}$ is uniformly bounded in
    $H^1(\Omega)$. Moreover, there exists a subsequence
    $\{n_j\}_{j\ge1}$ such that
    \[
    v^{n_j}\to0,\qquad v^{n_j+1}\to0
    \qquad \text{strongly in }H^1(\Omega).
    \]
    \item[(ii)] The sequence \(\{u^n\}_{n\in\mathbb N}\) is uniformly bounded in \(H^2(\Omega)\).
\end{itemize}
\end{lemma}

\begin{proof}
We first verify that the discrete iterates have the regularity required
below. For the GL model, \(\mathcal A_{\tau,\eta}=(1+\tau\eta)I-\tau^2\Delta\)
is a second-order elliptic operator. Standard elliptic regularity and
the decoupled formulation in Section~\ref{subsec:relaxed-decoupled}
therefore yield inductively \(u^n\in H^2(\Omega)\) for every \(n\ge0\). The first equation of \eqref{eq:relaxed-sav-scheme} then gives
\(v^n\in H^2(\Omega)\) for every \(n\ge1\).

For part \textnormal{(i)}, since
\[
\bar{\mathcal E}^n=(r^n)^2+\frac12\|v^n\|_{L^2}^2,
\]
the bound \(\bar{\mathcal E}^n\le \bar{\mathcal E}^0\) implies that \(\{v^n\}_{n\in\mathbb N}\) is uniformly bounded in \(L^2(\Omega)\).

Moreover, by the dissipation estimate \eqref{eq:discrete_energy_law-relax},
\[
(1-\vartheta)\mathcal D(v^{n+1},v^n)\le \bar{\mathcal E}^n-\bar{\mathcal E}^{n+1},
\]
with
\[
\mathcal D(v^{n+1},v^n)
=
\tau_{n+1}\eta^{n+1}\|v^{n+1}\|_{L^2}^2
+\tau_{n+1}^2\|\nabla v^{n+1}\|_{L^2}^2
+\frac12\|v^{n+1}-v^n\|_{L^2}^2.
\]
Summing over $n$ and using $\tau_{n+1}\ge\tau_{\min}>0$, we obtain
\[
\sum_{n=0}^{\infty}
\left(
\eta^{n+1}\|v^{n+1}\|_{L^2}^2
+
\|\nabla v^{n+1}\|_{L^2}^2
+
\|v^{n+1}-v^n\|_{L^2}^2
\right)
<\infty.
\]
Hence $\|\nabla v^n\|_{L^2}\to0$ and $\|v^{n+1}-v^n\|_{L^2}\to0$. Together with the uniform $L^2$ bound, this also shows that $\{v^n\}_{n\ge1}$ is uniformly bounded in $H^1(\Omega)$.

Moreover, using $\eta^{n+1}\ge\omega_0/(n+1)$, we have
\[
\sum_{n=0}^{\infty}\frac{\|v^{n+1}\|_{L^2}^2}{n+1}<\infty.
\]
Since the harmonic series diverges, there exists a subsequence $\{n_j\}$ such that $v^{n_j+1}\to0$ strongly in $L^2(\Omega)$. Since $\|v^{n+1}-v^n\|_{L^2}\to0$, we also have $v^{n_j}\to0$ strongly in $L^2(\Omega)$. Combining this with $\|\nabla v^n\|_{L^2}\to0$, we conclude that
\[
v^{n_j}\to0,\qquad v^{n_j+1}\to0
\qquad\text{strongly in }H^1(\Omega).
\]
This proves part \textnormal{(i)}.

For part~\textnormal{(ii)}, Lemma~\ref{lem:universal-physical-energy-bound}
yields a uniform upper bound for \(E(u^n)\).
Proposition~\ref{prop:GL-basic-estimates}\textnormal{(i)} therefore gives
\[
\sup_{n\in\mathbb N}\|u^n\|_{H^1}<\infty.
\]
By Proposition~\ref{prop:GL-basic-estimates}\textnormal{(ii)},
\(\{f(u^n)\}_{n\in\mathbb N}\) is thus uniformly bounded in
\(L^2(\Omega)\).

The second equation of \eqref{eq:relaxed-sav-scheme} can be written as
\[
-\Delta u^{n+1}
=
(1-a^n)(-\Delta u^n)-a^n f(u^n)
-\frac{v^{n+1}-v^n}{\tau_{n+1}}
-\eta^{n+1}v^{n+1},
\]
where \(a^n\) is defined in Hypothesis~\ref{hyp:ratio-control}. By Hypothesis~\ref{hyp:ratio-control}, there exist \(N\in\mathbb N\) and \(\delta\in(0,1)\) such that
\[
|1-a^n|\le 1-\delta,
\qquad
|a^n|\le 2-\delta
\qquad \text{for all } n\ge N.
\]

By part \textnormal{(i)}, \(\{v^n\}_{n\in\mathbb N}\) is uniformly bounded in \(L^2(\Omega)\). Using also \(\tau_{n+1}\ge \tau_{\min}>0\) and \(\eta^{n+1}\le \omega_1\), the last three terms on the right-hand side are uniformly bounded in \(L^2(\Omega)\). Consequently, there exists a constant \(M>0\), independent of \(n\), such that
\[
\|\Delta u^{n+1}\|_{L^2}
\le
(1-\delta)\|\Delta u^n\|_{L^2}+M
\qquad \text{for all } n\ge N.
\]
By induction, for all $n\ge N$,
\[
\|\Delta u^n\|_{L^2}
\le
(1-\delta)^{n-N}\|\Delta u^N\|_{L^2}
+
M\sum_{j=0}^{n-N-1}(1-\delta)^j
\le
\|\Delta u^N\|_{L^2}+\frac{M}{\delta}.
\]
Hence $\{\Delta u^n\}_{n\in\mathbb N}$ is uniformly bounded in $L^2(\Omega)$. Then standard elliptic regularity yields the uniform \(H^2(\Omega)\)-bound on \(\{u^n\}\).
This completes the proof of part~(ii).
\end{proof}
\begin{remark}
    The lower bound \(\eta^{n+1}\ge {\omega_0}/{(n+1)}\) in
Lemma~\ref{lem:GL-boundedness} is used to obtain a subsequence of the
velocity converging to zero in \(L^2(\Omega)\), since the present choice
\(\mathcal S=-\Delta\)  yields only
\[
\langle \mathcal S v,v\rangle=\|\nabla v\|_{L^2}^2.
\]
If instead \(\mathcal S=I-\Delta\) is used, then \(\langle \mathcal S v,v\rangle=\|v\|_{H^1}^2\), and this lower bound is no longer needed. Similarly, for the CSH model considered in Section~\ref{subsec:CSH-convergence}, under Assumption~\ref{ass:CSH-domain-size}, the choice \(\mathcal S=-\Delta(\Delta+1)^2=\mathcal G\mathcal L\) gives
\[
\langle \mathcal S v,\mathcal G^{-1}v\rangle=
\|(\Delta+1)v\|_{L^2}^2
\ge c\|v\|_{H^2}^2
\]
for some \(c>0\) and all \(v\in\dot H^2(\Omega)\), so no analogous lower bound on the damping
sequence is required.
\end{remark}
\begin{theorem}[Subsequential convergence to stationary points for the GL model]
\label{thm:GL-convergence}
Suppose that Hypothesis~\ref{hyp:ratio-control} holds and that the assumptions of Lemma~\ref{lem:GL-boundedness} are satisfied. Then there exist a subsequence $\{n_j\}$ and $u^\ast\in H^2(\Omega)$ such that
\[
u^{n_j}\rightharpoonup u^\ast
\quad\text{weakly in }H^2(\Omega),
\qquad
u^{n_j}\to u^\ast
\quad\text{strongly in }H^1(\Omega).
\]
Moreover,
\[
\langle\nabla u^\ast,\nabla\phi\rangle
+
\frac{1}{\varepsilon^2}
\left\langle
(u^\ast)^3-u^\ast,\phi
\right\rangle
=0,
\qquad
\forall\,\phi\in V.
\]
Hence $u^\ast$ is a stationary point of the GL energy.
\end{theorem}

\begin{proof}
By Lemma~\ref{lem:GL-boundedness}, $\{u^n\}$ is uniformly bounded in $H^2(\Omega)$, and there exists a subsequence $\{n_j\}$ such that $v^{n_j},v^{n_j+1}\to0$ strongly in $H^1(\Omega)$. By the Banach--Alaoglu theorem and the compact embedding $H^2(\Omega)\hookrightarrow H^1(\Omega)$, after passing to a further subsequence if necessary,
\[
u^{n_j}\rightharpoonup u^\ast
\quad\text{weakly in }H^2(\Omega),
\qquad
u^{n_j}\to u^\ast
\quad\text{strongly in }H^1(\Omega),
\]
for some $u^\ast\in H^2(\Omega)$. Moreover, since $u^{n_j+1}-u^{n_j}=\tau_{n_j+1}v^{n_j+1}$, the upper bound on the time step implies $u^{n_j+1}\to u^\ast$ strongly in $H^1(\Omega)$.

For $j$ sufficiently large, Hypothesis~\ref{hyp:ratio-control} gives $a^{n_j}\ge\delta$. Testing the second equation of the SF-SAV-E scheme \eqref{eq:relaxed-sav-scheme} with $\phi\in V$, using $\mathcal G=I$ and $\mathcal L=-\Delta$, and rearranging, we obtain
\[
\begin{aligned}
\langle\nabla u^{n_j},\nabla\phi\rangle
+\langle f(u^{n_j}),\phi\rangle
={}&
-\frac{1}{a^{n_j}}
\left\langle
\nabla(u^{n_j+1}-u^{n_j}),\nabla\phi
\right\rangle\\
&-
\frac{1}{a^{n_j}}
\left\langle
\frac{v^{n_j+1}-v^{n_j}}{\tau_{n_j+1}}
+\eta^{n_j+1}v^{n_j+1},
\phi
\right\rangle .
\end{aligned}
\]
Since $a^{n_j}\ge\delta$, the factors $1/a^{n_j}$ are uniformly bounded. The bounds on $\tau_{n+1}$ and $\eta^{n+1}$, together with $u^{n_j+1}-u^{n_j}\to0$ in $H^1(\Omega)$ and $v^{n_j},v^{n_j+1}\to0$ in $L^2(\Omega)$, show that the right-hand side tends to zero. Moreover, the strong $H^1$ convergence and $H^1(\Omega)\hookrightarrow L^6(\Omega)$ imply $f(u^{n_j})\to f(u^\ast)$ strongly in $L^2(\Omega)$. Passing to the limit yields
\[
\langle\nabla u^\ast,\nabla\phi\rangle
+
\frac{1}{\varepsilon^2}
\left\langle
(u^\ast)^3-u^\ast,\phi
\right\rangle
=0,
\qquad
\forall\,\phi\in V.
\]
This completes the proof.
\end{proof}

\subsection{Subsequential convergence to stationary points for CSH model}\label{subsec:CSH-convergence}
We next consider the CSH model, which is widely studied in materials science in connection with phase-field-crystal-type pattern formation~\cite{EldGra04,EldKatHaaGra02}. Under periodic boundary conditions, let \(\Omega=[0,L]^d\) with \(d\le 3\). Its energy functional is
\begin{equation}
E(u)=\frac12\int_{\Omega}(\Delta u+u)^2\,d\bx
+\int_{\Omega}\left(-\frac{\sigma}{2}u^2+\frac14 u^4\right)\,d\bx,
\label{eq:SH-energy-convergence}
\end{equation}
where $\sigma>0$ is a model parameter, with the mass constraint $\int_{\Omega}u\,d\bx=\mathrm{constant}$.
In this case,
\[
\mathcal L=(\Delta+1)^2,\qquad \mathcal G=-\Delta,\qquad f(u)=u^3-\sigma u.
\]
We take
\[ \mathcal S=\mathcal G\mathcal L=-\Delta(\Delta+1)^2,
\] 
as in the numerical experiments. For the convergence analysis below, we restrict to the zero-mean setting and use the periodic Sobolev subspaces
\[
\dot H^k(\Omega):=\left\{w\in H^k_{\mathrm{per}}(\Omega):\int_{\Omega}w\,d\bx=0\right\},
\qquad
\dot L^2(\Omega):=\dot H^0(\Omega).
\]
We denote by $P_0 g:=g-\frac{1}{|\Omega|}\int_\Omega g\,d\mathbf x$ the $L^2$-orthogonal projection onto \(\dot L^2(\Omega)\). On the zero-mean space, $-\Delta:\dot{H}^2(\Omega)\to\dot L^2(\Omega)$ is self-adjoint, positive definite, and invertible. Moreover,
\[
(-\Delta)^{-1}(-\Delta g)=P_0g,
\qquad g\in H^2_{\rm per}(\Omega).
\]

The convergence analysis follows the same overall strategy as in Section~\ref{subsec:GL-convergence}, with the additional zero-mean structure described above. We establish the basic estimates below, while the proofs of boundedness and subsequential convergence to stationary points are deferred to the Appendix.
\begin{assumption}[CSH domain-size condition]
\label{ass:CSH-domain-size}
Let \(\Omega=[0,L]^d\) be a periodic box in \(\mathbb R^d\) (\(d\le 3\)). We assume that
\(\Delta+1\) is injective on \(\dot H^2(\Omega)\),
equivalently, 
\[
    L\,\notin\,
    \left\{2\pi |\mathbf{k}|:\, \mathbf{k}\in\mathbb Z^d\setminus\{\mathbf{0}\}\right\}.
\]
\end{assumption}

Thus no nonzero Fourier mode lies in the kernel of \(\Delta+1\), which
provides the coercivity estimates used below.

\begin{proposition}[Basic estimates for the CSH model]\label{prop:CSH-basic-estimates}
Let \(E\) be given by \eqref{eq:SH-energy-convergence}. Suppose Assumption~\ref{ass:CSH-domain-size} holds. Then the following statements hold.

\begin{itemize}
\item[(i)] \textbf{Operator coercivity on the zero-mean subspace.}
    There exist constants \(c_2,c_4>0\) such that
    \begin{align*}
      \|(\Delta+1)u\|_{L^2}^2
    &\ge c_2 \|u\|_{H^2}^2
    \qquad \text{for all } u\in \dot H^2(\Omega), \\
    \|\mathcal L u\|_{L^2}
    &\ge c_4 \|u\|_{H^4}
    \qquad \text{for all } u\in \dot H^4(\Omega).
    \end{align*}
    
    \item[(ii)] \textbf{Energy coercivity.}
    \[
    E(u)+\frac{\sigma^2|\Omega|}{4}
    \ge
    \frac{c_2}{2}\|u\|_{H^2}^2,
    \qquad \text{for all }
    u\in\dot H^2(\Omega).
    \]

    \item[(iii)] \textbf{Nonlinear mapping estimate.}
    If \(\{u^n\}_{n\in\mathbb N}\) is uniformly bounded in \(H^2(\Omega)\), then
    \(\{f(u^n)\}_{n\in\mathbb N}\) is uniformly bounded in \(L^2(\Omega)\).
\end{itemize}
\end{proposition}

\begin{proof}
For part~\textnormal{(i)}, write
\[
u(\bx)=\sum_{\mathbf{k}\in\mathbb Z^d\setminus\{\mathbf{0}\}}\widehat u_{\mathbf{k}}
e^{i(2\pi \mathbf{k}/L)\cdot\bx},
\qquad
\lambda_{\mathbf{k}}=\left|\frac{2\pi {\mathbf{k}}}{L}\right|^2.
\]
Here $\widehat u_{\mathbf{k}}=|\Omega|^{-1}\int_\Omega u(\bx)e^{-i(2\pi\mathbf{k}/L)\cdot\bx}\,d\bx$.
Since $\lambda_{\mathbf{k}}\neq1$ for all
$\mathbf{k}\in\mathbb Z^d\setminus\{\mathbf{0}\}$ by Assumption~\ref{ass:CSH-domain-size}, and $\lambda_{\mathbf{k}}\to\infty$ as $|{\mathbf{k}}|\to\infty$, there exist constants $\tilde{c}_2,\tilde{c}_4>0$ such that
\[
(1-\lambda_{\mathbf{k}})^2
\ge \tilde{c}_2(1+\lambda_{\mathbf{k}})^2
\quad\mbox{and}\quad
(1-\lambda_{\mathbf{k}})^4
\ge \tilde{c}_4^2(1+\lambda_{\mathbf{k}})^4
\qquad\mbox{for all } \mathbf{k}\in\mathbb Z^d\setminus\{\mathbf{0}\}.
\]
Hence, by Parseval's identity and the equivalence of Fourier and Sobolev norms, for some constants $c_2,c_4>0$,
{\small\begin{align*}
 \|(\Delta+1)u\|_{L^2}^2
&=|\Omega|\sum_{{\mathbf{k}}\neq{\mathbf{0}}}(1-\lambda_{\mathbf{k}})^2|\widehat u_{\mathbf{k}}|^2
\ge \tilde{c}_2|\Omega|\sum_{{\mathbf{k}}\neq{\mathbf{0}}}(1+\lambda_{\mathbf{k}})^2|\widehat u_{\mathbf{k}}|^2
\ge c_2\|u\|_{H^2}^2,
\qquad u\in\dot H^2(\Omega), \\
\|\mathcal L u\|_{L^2}^2
&=|\Omega|\sum_{{\mathbf{k}}\neq{\mathbf{0}}}(1-\lambda_{\mathbf{k}})^4|\widehat u_{\mathbf{k}}|^2
\ge \tilde{c}_4^2|\Omega|\sum_{{\mathbf{k}}\neq{\mathbf{0}}}(1+\lambda_{\mathbf{k}})^4|\widehat u_{\mathbf{k}}|^2
\ge c_4^2\|u\|_{H^4}^2,
\qquad u\in\dot H^4(\Omega).
\end{align*}}
This proves part~\textnormal{(i)}.

For part~\textnormal{(ii)}, using part~\textnormal{(i)}, we have
\[
E(u)
=\frac12 \|(\Delta+1)u\|_{L^2}^2+\frac14\int_{\Omega}\left(u^2-\sigma\right)^2d\bx-\frac{\sigma^2|\Omega|}{4}
\ge\frac{c_2}{2}\|u\|_{H^2}^2-\frac{\sigma^2|\Omega|}{4}.
\]
This proves part~\textnormal{(ii)}.

For part \textnormal{(iii)}, since \(d\le3\), the Sobolev embedding \(H^2(\Omega)\hookrightarrow L^\infty(\Omega)\) holds. Thus a uniform \(H^2(\Omega)\)-bound on \(\{u^n\}\) implies that \(\{f(u^n)\}\) is uniformly bounded in \(L^2(\Omega)\).
\end{proof}

\begin{lemma}[Mass preservation and conditional boundedness for CSH model]
\label{lem:CSH-boundedness}
Let \(\{u^n,v^n,r^n\}\) be the sequence generated by the adaptive relaxed SF-SAV-E scheme for the CSH energy \eqref{eq:SH-energy-convergence}. Suppose Assumption~\ref{ass:CSH-domain-size} and Hypothesis~\ref{hyp:ratio-control} hold. Assume that \(u^0\in \dot H^6(\Omega)\), \(v^0\in \dot L^2(\Omega)\), and \(\eta^{n+1}\le \eta_{\max}\) for all \(n\ge 0\). Then the following statements hold.
\begin{itemize}

\item[(i)] The zero mean is preserved, i.e.,
\[
\int_\Omega u^n\,d\bx=0,
\qquad
\int_\Omega v^n\,d\bx=0
\qquad \text{for all } n\in\mathbb N.
\]

    \item[(ii)] The sequence \(\{v^n\}_{n\in\mathbb N}\) is uniformly bounded in \(\dot L^2(\Omega)\), and
    \[
    \sum_{n=1}^{\infty}\|v^n\|_{H^2}^2<\infty.
    \]

    \item[(iii)] The sequence \(\{u^n\}_{n\in\mathbb N}\) is uniformly bounded in \(\dot H^4(\Omega)\).
\end{itemize}
\end{lemma}
\begin{proof}
See Appendix~\ref{proof-lem:CSH-boundedness}.
\end{proof}

\begin{theorem}[Subsequential convergence to stationary points for CSH model]
\label{thm:CSH-convergence}
Let \(\{u^n,v^n,r^n\}\) be the sequence generated by the adaptive relaxed SF-SAV-E scheme for the CSH energy \eqref{eq:SH-energy-convergence}. Suppose Assumption~\ref{ass:CSH-domain-size} holds and the hypotheses of Lemma~\ref{lem:CSH-boundedness} are satisfied. Then there exist a subsequence \(\{n_j\}\) and a limit \(u^\ast\in \dot H^4(\Omega)\) such that
\[
u^{n_j}\to u^\ast \quad \text{strongly in } H^3(\Omega),
\]
\[
u^{n_j+1}\to u^\ast,\qquad
v^{n_j}\to0,\qquad
v^{n_j+1}\to0
\quad \text{strongly in } H^2(\Omega).
\]
Moreover, \(u^\ast\) is a stationary point of the CSH energy under the mass constraint, i.e.,
\[
\langle (\Delta+1)u^\ast,(\Delta+1)\phi\rangle
+\left\langle (u^\ast)^3-\sigma u^\ast,\phi\right\rangle=0
\qquad \text{for all } \phi\in \dot H^2(\Omega).
\]
\end{theorem}
\begin{proof}
See Appendix~\ref{proof-thm:CSH-convergence}.
\end{proof}

\section{Numerical experiments}\label{sect:numerics}
In this section, we present numerical experiments for the GL and CSH models. All simulations use Fourier pseudo-spectral discretizations with periodic boundary conditions. Unless otherwise specified, we use $N=128$ grid points in each spatial direction and all methods within each comparison use the same initial condition. For the second-order flows, we set $v^0=0$, and all relaxed SAV-based schemes use $\vartheta=0.99$. For the adaptive schemes, we set $\tau_{\max}=200$; this upper bound is inactive in all reported runs. All computations were implemented in Python 3.9 and performed on a computer equipped with a 12th Gen Intel Core i7-1260P processor. The stopping criteria and parameters specific to individual experiments are given when needed. For a discrete state $u^n$, we define the stationary residual by
\[
R_u^n:=
\left\|
\mathcal G\frac{\delta E}{\delta u}(u^n)
\right\|_{L^\infty}.
\]
For the second-order flows, we additionally use $R_v^n:=\|v^n\|_{L^\infty}$ as the velocity residual.

\subsection{Ginzburg--Landau model}\label{subsec:GL}

Throughout this subsection, we consider the GL energy \eqref{eq:GL-energy-convergence} on $\Omega=[0,2\pi]^2$. Unless otherwise specified, we take $\varepsilon=0.1$ and use $\eta(t)=\alpha/t$ with $\alpha=100$. Random initial data are generated independently at each grid point from the standard uniform distribution. For the GL model, the stationary residual becomes $R_u^n:=\|-\Delta u^n+\varepsilon^{-2}((u^n)^3-u^n)\|_{L^\infty}$. In all converged runs using these random initial data, the computed solution approaches a zero-energy state.

For comparison, GF-SAV-E~\cite{LiuSheZha23} and GF-SAV-F~\cite{SheXuYan18,SheXuYan19} denote the corresponding gradient-flow SAV-E and SAV-F schemes, respectively, while SF-SAV-F denotes the second-order SAV-F scheme \eqref{eq:sf-sav-f-scheme}. For these SAV-based schemes, we adopt relaxation procedures following \cite{JiaZhaZha22,LiuSheZha23},
with the relaxation parameters selected according to the corresponding discrete energy dissipation laws.

\begin{example}[Effect of relaxation]\label{eg:4-1}
We first illustrate the role of the relaxation step in the SF-SAV-E scheme. For the GL energy \eqref{eq:GL-energy-convergence}, $E(u)\ge0$, and the constant states $u\equiv\pm1$ attain $E=0$. Thus, the zero-energy level provides a natural reference for this experiment. We compare the relaxed SF-SAV-E scheme with the same scheme without relaxation using a fixed time step $\tau=0.1$ and $C=10^{-12}$.

Figure~\ref{fig:comparison_relaxation} shows the evolution of the energy and the auxiliary variable. Without relaxation, $(r^n)^2$ decays much faster than $E(u^n)+C$ and becomes very small while the energy is still far from its minimum, causing the iteration to stagnate. With relaxation, $(r^n)^2$ remains close to $E(u^n)+C$, and $E(u^n)+C$ decreases toward $C=10^{-12}$, corresponding to $E(u^n)\to0$.
    \begin{figure}[!htbp]
        \centering
        \includegraphics[height=0.25\textheight]{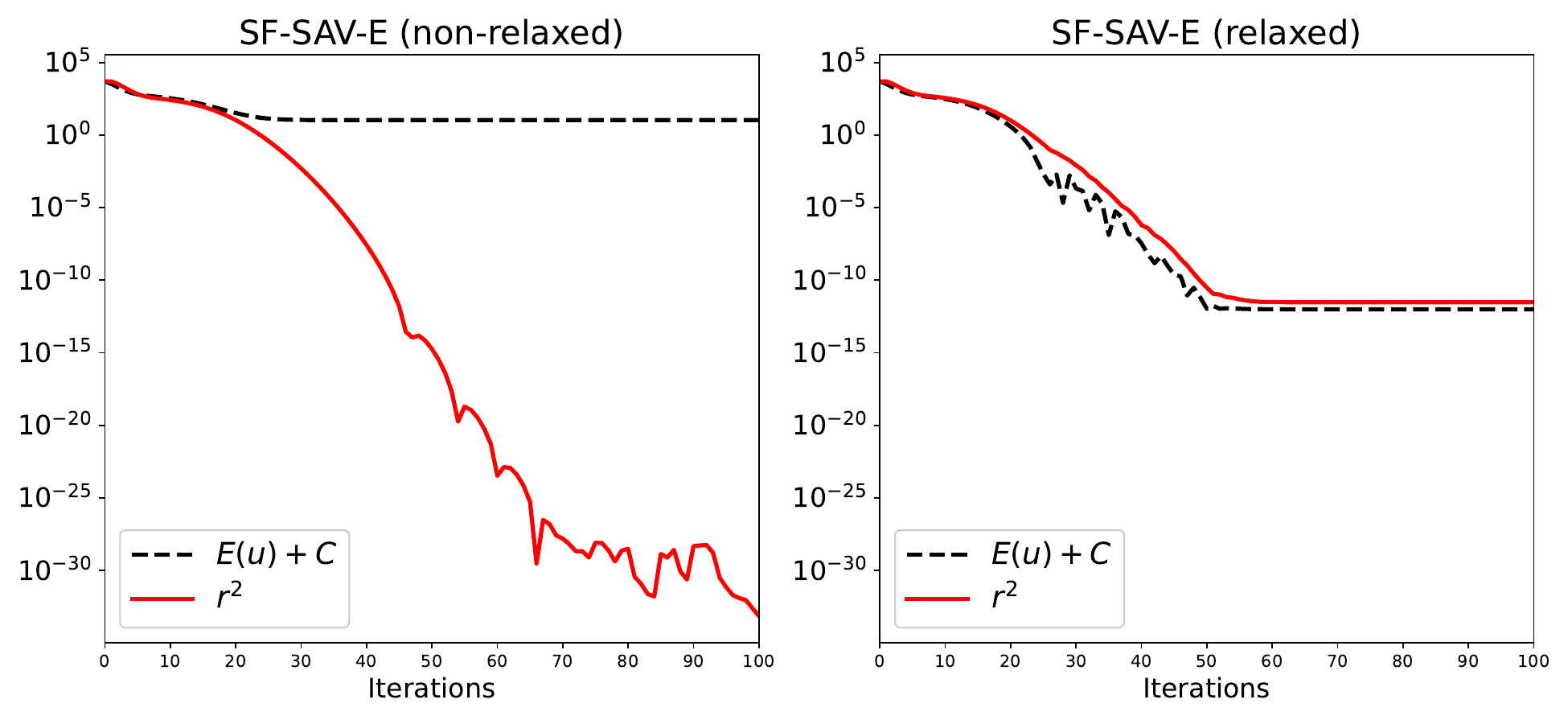}
        \caption{Comparison of the energy $E(u)+C$ and the squared auxiliary variable $r^2$ for the SF-SAV-E algorithm without (left) and with (right) the relaxation step, over 100 iterations.}
        \label{fig:comparison_relaxation}
    \end{figure}
\FloatBarrier
\end{example}

\begin{example}[Comparison of SAV-F and SAV-E]\label{eg:4-2}
This example examines the stationary-state consistency of the SAV-F and SAV-E formulations for both the gradient flow and the second-order flow. As seen from \eqref{eq:sav-f-stationary} and \eqref{eq:sav-e-stationary}, the auxiliary ratio multiplies only the nonlinear term in SAV-F, but the entire variational derivative in SAV-E. We therefore choose $u^0(x_1,x_2)=0.05\sin x_1\sin x_2$, which produces nonconstant limiting states with $f(u^\star)\not\equiv0$ in our computations, so that the effect of the SAV-F auxiliary ratio remains visible. All four schemes use the fixed time step $\tau=0.1$. Since the GL potential and energy are non-negative, we set $C_F=C_E=10^{-12}$ and initialize $r_F^0=\sqrt{\int_\Omega F(u^0)\,dx+C_F}$ for SAV-F and $r^0=\sqrt{E(u^0)+C_E}$ for SAV-E.
The iterations are stopped when $\|u^{n+1}-u^n\|_{L^\infty}<10^{-12}$. For the second-order schemes, we additionally require $\|v^{n+1}-v^n\|_{L^\infty}<10^{-10}$. Table~\ref{tab:sav_f_vs_e_final} reports the final energy, original Euler--Lagrange residual, SAV-F residual corresponding to \eqref{eq:sav-f-stationary}, and terminal auxiliary ratio.

Table~\ref{tab:sav_f_vs_e_final} illustrates the distinction discussed in Section~2.2. For both SAV-F schemes, the terminal auxiliary ratio differs from one, and the residual of the modified stationary condition \eqref{eq:sav-f-stationary} is very small, while the residual of the original Euler--Lagrange equation remains non-negligible. In contrast, both SAV-E schemes attain small Euler--Lagrange residuals. In particular, GF-SAV-E reaches an Euler--Lagrange residual $\left\|{\delta E}/{\delta u}\right\|_{L^\infty}<10^{-9}$ even with a terminal auxiliary ratio of about $0.11$, showing that SAV-E does not require this ratio to be close to one.
\begin{table}[!htbp]
\centering
\caption{Final energy, stationary residuals, and terminal auxiliary ratio for the compared schemes. The SAV-F residual is the $L^\infty$-norm of the left-hand side of \eqref{eq:sav-f-stationary}. The terminal auxiliary ratio is $\tilde r_F^{n+1}/\sqrt{\int_\Omega F(u^n)\,dx+C_F}$ for SAV-F and $\tilde r^{n+1}/\sqrt{E(u^n)+C_E}$ for SAV-E.}
\label{tab:sav_f_vs_e_final}
\begin{tabular}{lcccc}
\hline
Algorithm
& Energy
& $\left\|{\delta E}/{\delta u}\right\|_{L^\infty}$
& SAV-F residual
& Auxiliary ratio \\
\hline
GF-SAV-F
& 262.04
& $33.35$
& $1.66\times10^{-11}$
& 0.13 \\
SF-SAV-F
& 167.77
& $4.08$
& $1.15\times10^{-10}$
& 1.11 \\
GF-SAV-E
& 167.55
& $9.57\times10^{-10}$
& --
& 0.11 \\
SF-SAV-E
& 167.55
& $1.04\times10^{-10}$
& --
& 1.00 \\
\hline
\end{tabular}
\end{table}
\FloatBarrier
\end{example}

\begin{example}[Numerical evidence for Hypothesis~\ref{hyp:ratio-control}]\label{eg:hypothesis-ratio-evidence}\label{ex:adaptive-ratio-gl}
To examine Hypothesis~\ref{hyp:ratio-control} numerically, we use an adaptive setting of SF-SAV-E with $\tau_0=\tau_{\min}=0.01$, $\rho=1.01$, and vary $I_{\rm thr}\in\{0.6,0.7,0.75,0.8,0.9,0.95\}$. Each run is continued for 3000 updates. Figure~\ref{fig:example53-adaptive-a} shows the evolution of $a^n$ for $I_{\rm thr}=0.6,0.75,0.9$, while Figure~\ref{fig:example53-adaptive-b} reports its minimum and maximum over the late-time window $n=2000,\ldots,2999$ for all tested values of $I_{\rm thr}$. For every tested value of $I_{\rm thr}$, these late-time bounds remain strictly inside $(0,2)$, providing numerical evidence consistent with Hypothesis~\ref{hyp:ratio-control}. Moreover, the late-time values of \(a^n\) are close to the corresponding \(I_{\rm thr}\), suggesting a stable tracking behavior that helps explain why \(a^n\) remains bounded away from \(0\) and \(2\).  

\begin{figure}[H]
\centering
\subfloat[Evolution of $a^n$ for $I_{\rm thr}=0.6,0.75,0.9$.\label{fig:example53-adaptive-a}]{%
    \includegraphics[height=0.21\textheight]{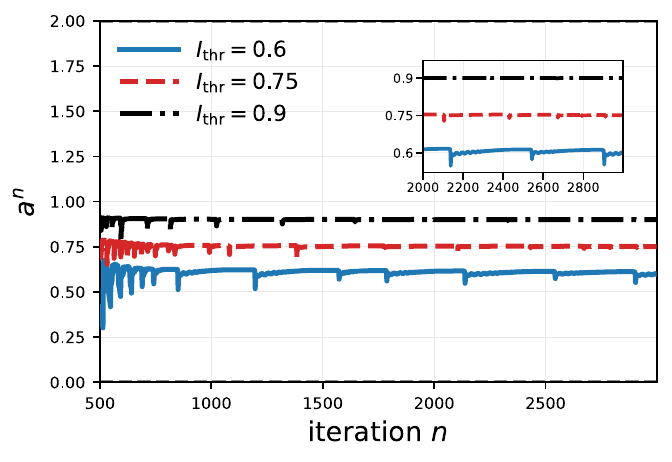}}%
\hfill
\subfloat[Windowed bounds of $a^n$ versus $I_{\rm thr}$.\label{fig:example53-adaptive-b}]{%
    \includegraphics[height=0.21\textheight]{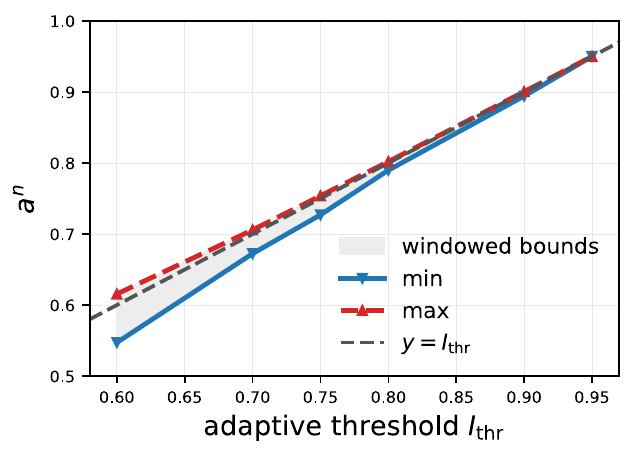}}%
\caption{Numerical evidence for Hypothesis~\ref{hyp:ratio-control}. Panel (a) shows the evolution of $a^n$ for $I_{\rm thr}=0.6,0.75,0.9$. Panel (b) shows the minimum and maximum of $a^n$ over $n=2000,\ldots,2999$ for all tested values of $I_{\rm thr}$.}
\label{fig:example53-adaptive}
\end{figure}

\end{example}

\begin{example}[Comparison with other second-order flow schemes]\label{eg:4-3}
We compare fixed and adaptive SF-SAV-E with the stabilized semi-implicit
scheme SF-SSI~\cite{CheDonLiuXie23} and the convex-splitting scheme
SF-CS~\cite{CheDonIglLiuXie25}. We take $N=256$ and $\varepsilon=0.1$. For all second-order flow schemes, the damping is $\eta(t)=\alpha/t$, with $\alpha\in\{500,1000,5000\}$. For the fixed-step schemes, we consider
$\tau\in\{0.01,0.1,1,5\}$, while the same four values are used as
the initial step sizes $\tau_0$ for adaptive SF-SAV-E. For the adaptive scheme, we take $\tau_{\min}=0.01$, $\rho=1.01$, and $I_{\rm thr}=0.75$. The same random initial field is used for all runs. All runs are terminated when $R_u^n<10^{-8}$ and $R_v^n<10^{-8}$, with at most 10000 updates.

Table~\ref{tab:save_adaptive_counts} shows that, for fixed SF-SAV-E, the number of iterations depends strongly on the choices of $\tau$ and $\alpha$. A similar sensitivity to the damping parameter was observed for SF-SSI in~\cite{CheDonLiuXie23}. Over the tested parameter combinations, fixed SF-SAV-E requires $43$--$4489$ iterations, whereas adaptive SF-SAV-E requires $41$--$372$. Although a well-chosen fixed step can still be more efficient in an individual case, adaptive time stepping dramatically reduces the sensitivity to the damping parameter and the initial time step.

Table~\ref{tab:second_order_times} reports the computational times required to reach the convergence criterion. SF-SSI does not reach the prescribed tolerance in several fixed-step cases, while SF-CS requires a nonlinear Newton--CG solve at each step. Adaptive SF-SAV-E retains a linear and decoupled update. It is faster than SF-CS in all tested cases and faster than SF-SSI whenever SF-SSI converges.

\begin{table}[!htbp]
\centering
\caption{Iterations to reach the convergence criterion for fixed and adaptive
SF-SAV-E with $\eta(t)=\alpha/t$.}
\label{tab:save_adaptive_counts}

\small
\begin{tabular*}{0.75\textwidth}{
@{\extracolsep{\fill}}
c
*{8}{c}
@{}
}
\toprule
&
\multicolumn{4}{c}{\normalsize Fixed SF-SAV-E ($\tau$)}
&
\multicolumn{4}{c}{\normalsize Adaptive SF-SAV-E ($\tau_0$)}
\\
\cmidrule(lr){2-5}\cmidrule(lr){6-9}
$\alpha$
& 0.01 & 0.1 & 1 & 5
& 0.01 & 0.1 & 1 & 5
\\
\midrule
500
& 766 & 97 & 222 & 4489
& 233 & 82 & 309 & 200
\\
1000
& 1109 & 128 & 65 & 2617
& 265 & 95 & 372 & 270
\\
5000
& 2502 & 273 & 43 & 617
& 340 & 140 & 41 & 173
\\
\bottomrule
\end{tabular*}
\end{table}

\begin{table}[!htbp]
\centering
\caption{CPU times (s) to reach the convergence criterion for the second-order flow schemes with $\eta(t)=\alpha/t$. ``Not converged'' indicates that the stopping criterion was not reached within 10000 updates.}
\label{tab:second_order_times}
\small
\setlength{\tabcolsep}{3pt}
\begin{tabular}{ccrrrr}
\toprule
$\alpha$ & $\tau/\tau_0$ & SF-SSI & SF-CS & \shortstack{SF-SAV-E\\(fixed)} & \shortstack{SF-SAV-E\\(adaptive)} \\
\midrule
\multirow{4}{*}{500} & 0.01 & 15.01 & 67.72 & 17.03 & 5.49 \\
 & 0.1 & Not converged & 47.61 & 2.07 & 2.14 \\
 & 1 & Not converged & 28.15 & 4.79 & 8.45 \\
 & 5 & Not converged & 26.87 & 105.21 & 4.50 \\
\midrule
\multirow{4}{*}{1000} & 0.01 & 21.53 & 87.87 & 27.06 & 6.63 \\
 & 0.1 & 63.10 & 49.27 & 2.96 & 2.46 \\
 & 1 & Not converged & 32.32 & 1.66 & 9.17 \\
 & 5 & Not converged & 27.12 & 59.83 & 6.41 \\
\midrule
\multirow{4}{*}{5000} & 0.01 & 48.26 & 190.23 & 55.57 & 8.18 \\
 & 0.1 & 4.77 & 71.47 & 6.00 & 3.14 \\
 & 1 & Not converged & 31.12 & 0.97 & 0.95 \\
 & 5 & Not converged & 25.33 & 13.27 & 3.78 \\
\bottomrule
\end{tabular}
\end{table}

\FloatBarrier
\end{example}

\begin{example}[Comparison with gradient flow]\label{eg:4-4}
We finally compare SF-SAV-E with its gradient-flow counterpart, GF-SAV-E, to examine the effect of the second-order dynamics. We use the smooth deterministic sign-changing initial condition
\[
u^0(x_1,x_2)
=
0.1+0.25(\sin x_1+\sin x_2)
+0.1\sin(x_1+x_2).
\]
We fix $N=512$ and take $\varepsilon=1/k$ for $k=10,\ldots,30$. Both methods use the same adaptive time-stepping strategy with
$\tau_0=\tau_{\min}=0.01$, $\rho=1.01$, and $I_{\rm thr}=0.75$. Convergence is declared when \(R_u^n<10^{-8}\), and for SF-SAV-E we additionally require \(R_v^n<10^{-8}\). For every tested $\varepsilon$, both methods converge to the same positive homogeneous zero-energy state.
At the fixed spatial resolution $N=512$, Figure~\ref{fig:acceleration} shows that the two methods require comparable numbers of iterations for relatively large $\varepsilon$. As $\varepsilon$ decreases, SF-SAV-E shows an increasingly pronounced advantage in the number of iterations required for convergence. At $1/\varepsilon=30$, this advantage is substantial: SF-SAV-E requires 1748 iterations, compared with 8212 for GF-SAV-E.
\begin{figure}[!htbp]
    \centering
    \includegraphics[width=0.68\linewidth]{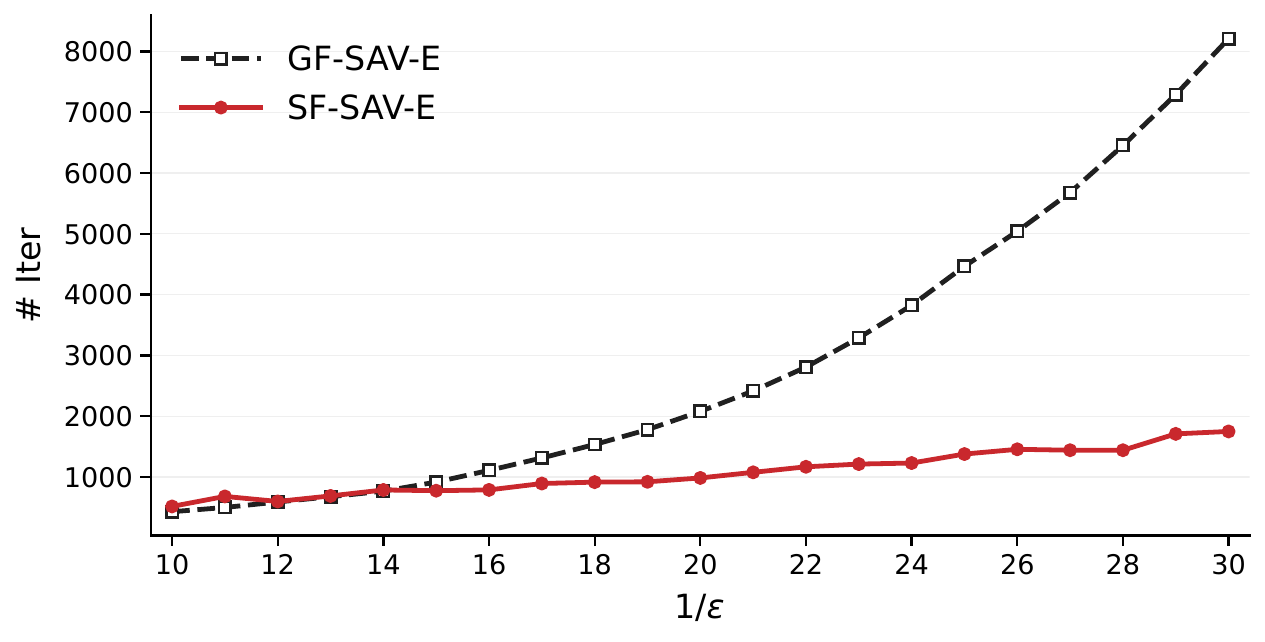}
    \caption{Number of iterations to convergence versus $1/\varepsilon$ for the GL model.}
    \label{fig:acceleration}
\end{figure}
\FloatBarrier
\end{example}

\subsection{Numerical results for the CSH model}\label{subsec:CSH}
 We now consider numerical experiments for the CSH model introduced in Section~\ref{subsec:CSH-convergence}. We take the energy \eqref{eq:SH-energy-convergence} on $\Omega=[0,50]^2$, with $\mathcal G=-\Delta$ for both gradient and second-order flows. The parameter \(\sigma>0\) serves as a reduced control parameter in the CSH/PFC model, which is often related to undercooling or to the distance from the instability threshold. As \(\sigma\) increases, the energy landscape typically becomes more strongly nonconvex, with enhanced metastability and more pronounced trapping by defect-containing states.

\begin{example}[Computational efficiency and converged states]\label{eg:4-5}
We compare SF-SAV-E with GF-SAV-E on the CSH model, employing the same adaptive time-stepping strategy for both methods with $\tau_{\min}=1$, $\rho=1.1$, $I_{\rm thr}=0.75$, and initial $\tau_0=1$. For SF-SAV-E, we take $\eta(t)=\alpha/t$ and test $\alpha\in\{50,100,200,300\}$ for each value of $\sigma$. Since the damping parameter influences the evolution of the second-order flow and may affect the converged state, we report the lowest-energy result among these runs. The corresponding value of $\alpha$ is listed in Table~\ref{tab:lb_results}. The reported iteration counts and CPU times correspond to the selected run. 

Both methods use the same random initial condition $u_{i,j}^0=0.01(2R_{i,j}-1)$, where the $R_{ij}$ are independent standard uniform random variables. The resulting field is then projected to zero mean. We take $C=100$. Since $\mathcal G=-\Delta$,
the stationary residual takes the form
\[
R_u^n=
\left\|
\Delta\frac{\delta E}{\delta u}(u^n)
\right\|_{L^\infty}.
\]
All runs are terminated when $R_u^n<10^{-8}$, and for SF-SAV-E
we additionally require $R_v^n<10^{-8}$.
\begin{table}[!ht]
    \centering
    \caption{Iterations to convergence, final energies, and computational times for GF-SAV-E and SF-SAV-E on the CSH model.}
    \label{tab:lb_results}
    \begin{tabular}{c|c|c|c|c|c}
        \hline
        \(\sigma\) & Algorithm & \(\eta(t)\) & Iterations & Energy & Time (s) \\
        \hline
        \multirow{2}{*}{0.1}
        & GF-SAV-E & --      & 21518 & -4.1586          & 356.13 \\
        & SF-SAV-E & \(300/t\) & 2152  & -4.1586          & 53.74 \\
        \hline
        \multirow{2}{*}{0.2}
        & GF-SAV-E & --      & 12053 & -16.4646         & 121.32 \\
        & SF-SAV-E & \(200/t\) & 1847  & \textbf{-16.6590} & 47.53 \\
        \hline
        \multirow{2}{*}{0.3}
        & GF-SAV-E & --      & 21746 & -37.5097         & 361.79 \\
        & SF-SAV-E & \(200/t\) & 1677  & -37.5097         & 42.75 \\
        \hline
        \multirow{2}{*}{0.4}
        & GF-SAV-E & --      & 23902 & -62.5890         & 252.31 \\
        & SF-SAV-E & \(50/t\)  & 5893  & \textbf{-66.7189} & 150.18 \\
        \hline
    \end{tabular}
\end{table}

Table~\ref{tab:lb_results} summarizes the results. For $\sigma=0.1$ and $0.3$, the two methods converge to states with essentially the same energy, while SF-SAV-E requires substantially fewer iterations and less computational time. For $\sigma=0.2$ and $0.4$, SF-SAV-E reaches lower-energy states than GF-SAV-E. 

Figure~\ref{fig:lb_patterns} shows the corresponding final states for all four values of $\sigma$. For $\sigma=0.1$, the two methods converge to stripe states with essentially the same energy. For $\sigma=0.2$, both methods produce stripe-like states, although the GF-SAV-E pattern is less geometrically regular. For $\sigma=0.3$, both methods produce regular stripe states. The most pronounced difference occurs for $\sigma=0.4$, where GF-SAV-E remains in a defect-containing stripe configuration, whereas SF-SAV-E reaches a regular parallel-stripe state with lower energy.

\begin{figure}[t!]
    \centering
    \includegraphics[width=\linewidth]{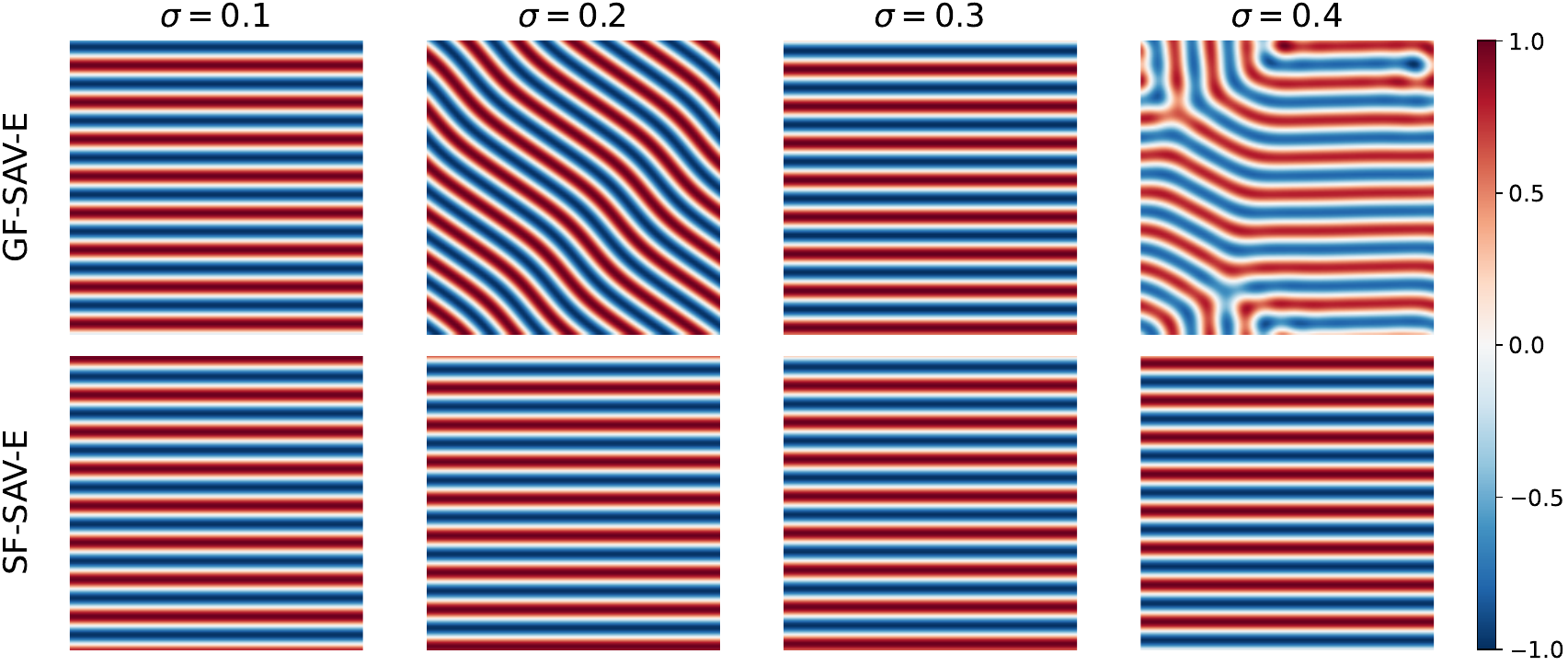}
    \caption{Final states obtained by GF-SAV-E (top row) and SF-SAV-E (bottom row); the columns correspond to $\sigma=0.1$, $0.2$, $0.3$, and $0.4$. Each field is normalized by its $L^\infty$-norm for visualization.}
    \label{fig:lb_patterns}
\end{figure}

\end{example}

\FloatBarrier
\section{Conclusion}\label{sect:conclusion}

{\looseness=-1
In this paper, we developed a robust and efficient SAV framework for approaching stationary points with second-order flows. We derived a linear and decoupled implementation of the relaxed scheme and established its unconditional modified-energy stability. Under suitable model assumptions and a ratio-control hypothesis, we further established subsequential convergence to stationary points for the Ginzburg--Landau and conserved Swift--Hohenberg models. Numerical experiments illustrate the roles of relaxation and adaptive time stepping, demonstrate the distinction between SAV-F and SAV-E in stationary-state consistency, and show clear computational advantages of SF-SAV-E over previous second-order flow schemes and the corresponding gradient flow in the tested problems.\par}

Future work includes extensions to other variational models and constrained problems, and stronger convergence theory.

\section*{Acknowledgments}
Generative AI tools were used to assist with language editing and preliminary checks of manuscript consistency, mathematical arguments, and references. The authors assume responsibility for all content.

\appendix

\section{Proof of Lemma~\ref{lem:CSH-boundedness}}\label{proof-lem:CSH-boundedness} 
We first verify the regularity of the discrete iterates.
For the CSH model, $\mathcal A_{\tau,\eta}=(1+\tau\eta)I-\tau^2\Delta(\Delta+1)^2$ is a sixth-order elliptic operator under periodic boundary conditions. Standard elliptic regularity together with the decoupled formulation in
Section~\ref{subsec:relaxed-decoupled} therefore yield inductively $u^n\in H^6_{\rm per}(\Omega)$ for every $n\ge0$. The first equation of \eqref{eq:relaxed-sav-scheme} then gives $v^n\in H^6_{\rm per}(\Omega)$ for every $n\ge1$.

For part \textnormal{(i)}, recall that $\mathcal G=-\Delta$ and $\mathcal S=-\Delta(\Delta+1)^2$ for the CSH model.
Therefore, by periodicity,
\begin{align}\label{eq:prop-CSH-basic-intSGL}
\int_\Omega \mathcal S u^{n+1}\,d\mathbf{x}
=
\int_\Omega \mathcal S u^n\,d\mathbf{x}
=
\int_\Omega
\mathcal G\bigl(\mathcal L u^n+f(u^n)\bigr)\,d\mathbf{x}
=0,
\end{align}
since each term is the integral of a Laplacian over the periodic domain. Integrating the first two equations of \eqref{eq:relaxed-sav-scheme}
over $\Omega$ and applying \eqref{eq:prop-CSH-basic-intSGL}, one gets
\[
\begin{aligned}
&\int_\Omega u^{n+1}\,d\bx
=
\int_\Omega u^n\,d\bx
+
\tau_{n+1}\int_\Omega v^{n+1}\,d\bx.\\
&\frac{1}{\tau_{n+1}}
\int_\Omega (v^{n+1}-v^n)\,d\bx
+
\eta^{n+1}\int_\Omega v^{n+1}\,d\bx
=0.
\end{aligned}
\]
Since \(\int_\Omega v^0\,d\bx=0\), the second identity implies \(\int_\Omega v^n\,d\bx=0\) for all \(n\in\mathbb N\). The first identity, together with \(\int_\Omega u^0\,d\bx=0\), then gives \(\int_\Omega u^n\,d\bx=0\) for all \(n\in\mathbb N\).

For part \textnormal{(ii)}, since
\(\langle \mathcal S v,\mathcal G^{-1}v\rangle=\|(\Delta+1)v\|_{L^2}^2\), 
the dissipation estimate \eqref{eq:discrete_energy_law-relax} implies
\[
(1-\vartheta)\tau_{n+1}^2
\|(\Delta+1)v^{n+1}\|_{L^2}^2
\le
\bar{\mathcal E}^n-\bar{\mathcal E}^{n+1}.
\]
Summing from \(n=0\) to \(k-1\) and using \(\tau_{n+1}\ge\tau_{\min}\), we obtain
\[
(1-\vartheta)\tau_{\min}^2\sum_{n=0}^{k-1}
\|(\Delta+1)v^{n+1}\|_{L^2}^2
\le
\bar{\mathcal E}^{0}.
\]
By part~\textnormal{(i)} and
Proposition~\ref{prop:CSH-basic-estimates}\textnormal{(i)},
\(\sum_{n=1}^{\infty}\|v^n\|_{H^2}^2<\infty.\) 
In particular, \(v^n\to0\) strongly in \(H^2(\Omega)\), and
\(\{v^n\}_{n\in\mathbb N}\) is uniformly bounded in
\(\dot L^2(\Omega)\). This proves part~\textnormal{(ii)}.

For part~\textnormal{(iii)}, Lemma~\ref{lem:universal-physical-energy-bound}
yields a uniform upper bound for \(E(u^n)\).
By part~\textnormal{(i)} and the energy coercivity in
Proposition~\ref{prop:CSH-basic-estimates}\textnormal{(ii)},
\(\{u^n\}\) is uniformly bounded in \(\dot H^2(\Omega)\).
Proposition~\ref{prop:CSH-basic-estimates}\textnormal{(iii)}
therefore implies that \(\{f(u^n)\}\) is uniformly bounded in
\(L^2(\Omega)\).

Applying \((-\Delta)^{-1}\) to the second equation of the scheme \eqref{eq:relaxed-sav-scheme} gives
\[
\mathcal L u^{n+1}
=
(1-a^n)\mathcal L u^n
-a^n P_0f(u^n)
-
(-\Delta)^{-1}\!\left(
\frac{v^{n+1}-v^n}{\tau_{n+1}}+\eta^{n+1}v^{n+1}
\right),
\]
where \(a^n\) is defined in Hypothesis~\ref{hyp:ratio-control}. Since \(P_0\) is an \(L^2\)-orthogonal projection, we have $\|P_0f(u^n)\|_{L^2}\le \|f(u^n)\|_{L^2}$. Using also $|1-a^n|\le1-\delta$, $|a^n|\le2-\delta$,  $\tau_{n+1}\geq\tau_{\min}$, $\eta^{n+1}\leq\eta_{\max}$, together with part~\textnormal{(ii)} and the boundedness of
\((-\Delta)^{-1}\) on the zero-mean space, we obtain
\[
\|\mathcal L u^{n+1}\|_{L^2}
\le
(1-\delta)\|\mathcal L u^n\|_{L^2}+M
\qquad \text{for all } n\ge N
\]
for some constant \(M>0\). A standard induction then yields $\sup_{n\in\mathbb N}\|\mathcal L u^n\|_{L^2}<\infty$. Proposition~\ref{prop:CSH-basic-estimates}\textnormal{(i)} therefore implies
\(\sup_{n\in\mathbb N}\|u^n\|_{H^4}<\infty.\) 
Together with the zero-mean property in part~\textnormal{(i)}, this
shows that \(\{u^n\}\) is uniformly bounded in
\(\dot H^4(\Omega)\).
This completes the proof.

\section{Proof of Theorem~\ref{thm:CSH-convergence}}\label{proof-thm:CSH-convergence}
Lemma~\ref{lem:CSH-boundedness}\textnormal{(iii)} yields a uniform $H^4(\Omega)$ bound on $\{u^n\}$, while Lemma~\ref{lem:CSH-boundedness}\textnormal{(ii)} gives $v^n\to0$ strongly in $H^2(\Omega)$. By the Banach--Alaoglu theorem and the compact embedding $H^4(\Omega)\hookrightarrow H^3(\Omega)$, after passing to a subsequence $\{n_j\}$,
\[
u^{n_j}\rightharpoonup u^\ast
\quad\text{weakly in }H^4(\Omega),
\qquad
u^{n_j}\to u^\ast
\quad\text{strongly in }H^3(\Omega).
\]
By Lemma~\ref{lem:CSH-boundedness}\textnormal{(i)}, the zero-mean property is preserved for $\{u^{n_j}\}$, so the weak limit $u^\ast\in\dot H^4(\Omega)$. Since $u^{n_j+1}-u^{n_j}=\tau_{n_j+1}v^{n_j+1}$ and $\tau_{n_j+1}\le\tau_{\max}$, we also have $u^{n_j+1}\to u^\ast$ strongly in $H^2(\Omega)$.

For $j$ sufficiently large, Hypothesis~\ref{hyp:ratio-control} gives $a^{n_j}\ge\delta$. Applying $(-\Delta)^{-1}$ to the second equation of \eqref{eq:relaxed-sav-scheme}, testing with $\phi\in\dot H^2(\Omega)$, and using $\langle P_0f(u^{n_j}),\phi\rangle=\langle f(u^{n_j}),\phi\rangle$, we obtain
\[
\begin{aligned}
&
\langle(\Delta+1)u^{n_j},(\Delta+1)\phi\rangle
+
\langle f(u^{n_j}),\phi\rangle\\
&\qquad\qquad =
-\frac{1}{a^{n_j}}
\left\langle
(\Delta+1)(u^{n_j+1}-u^{n_j}),
(\Delta+1)\phi
\right\rangle\\
&\qquad\qquad\quad\,
-\frac{1}{a^{n_j}}
\left\langle
(-\Delta)^{-1}
\left(
\frac{v^{n_j+1}-v^{n_j}}{\tau_{n_j+1}}
+\eta^{n_j+1}v^{n_j+1}
\right),
\phi
\right\rangle .
\end{aligned}
\]
Since $a^{n_j}\ge\delta$, $\tau_{n_j+1}\geq\tau_{\min}$, $\eta^{n_j+1}\leq\eta_{\max}$, the strong limits $u^{n_j+1}-u^{n_j}\to0$ in $H^2(\Omega)$ and $v^{n_j},v^{n_j+1}\to0$ in $H^2(\Omega)$ show that the right-hand side tends to zero. The boundedness of $(-\Delta)^{-1}$ on the zero-mean space is used in the second term. The strong convergence $u^{n_j}\to u^\ast$ in $H^3$ also implies $f(u^{n_j})\to f(u^\ast)$ strongly in $L^2(\Omega)$. Passing to the limit gives
\[
\langle(\Delta+1)u^\ast,(\Delta+1)\phi\rangle
+\left\langle (u^\ast)^3-\sigma u^\ast,\phi\right\rangle=0,
\qquad \forall\,\phi\in\dot H^2(\Omega).
\]
Thus $u^\ast$ is a stationary point of the CSH energy under the mass constraint. This completes the proof.

\bibliographystyle{siamplain}
\bibliography{references}
\end{document}